\documentclass[11pt, a4paper]{article}
\usepackage{fullpage}
\usepackage[utf8]{inputenc}
\usepackage{footnote}
\usepackage{amsmath, amssymb, amsfonts, amsthm, mathtools, xcolor}
\usepackage{array}
\usepackage{systeme}
\usepackage{bm}
\usepackage{upgreek}
\usepackage[english]{babel}

\newcommand{\setbuilder}[2]{\left\{#1\ \colon #2\right\}}

\DeclareMathOperator{\diam}{diam}

\newcommand{\R}{{\mathbb R}}

\newcommand{\eps}{{\varepsilon}}

\newtheorem{theorem}{Theorem}
\newtheorem{claim}{Claim}

\newtheorem{lemma}{Lemma}

\newtheorem{conjecture}{Conjecture}
\theoremstyle{definition}
\newtheorem{definition}{Definition}
\newtheorem{remark}{Remark}
\newtheorem{thm}{Theorem}
\newtheorem{obs}{Observation}

\usepackage[T2A]{fontenc}

\title{On the Erd\H os--Purdy problem and the Zarankiewitz problem for semialgebraic graphs}

\author{N\'ora Frankl\thanks{Carnegie Mellon University, Pittsburgh, US and Moscow Institute of Science and Technology, Russia. Email: {\tt nfrankl@andrew.cmu.edu } 
} \and Andrey Kupavskii\thanks{G-SCOP, CNRS, University Grenoble-Alpes, France and Moscow Institute of Physics and Technology, Russia;  Email: {\tt kupavskii@yandex.ru} \ \  Research of Andrey Kupavskii is supported by the grant RSF N 21-71-10092, \url{https://rscf.ru/project/21-71-10092/}}}

\begin{document}
\date{}
\maketitle
\begin{abstract}
Erd\H os and Purdy, and later Agarwal and Sharir, conjectured that any set of $n$ points in $\R^{d}$ determine at most $Cn^{d/2}$ congruent $k$-simplices for even $d$. We obtain the first significant progress towards this conjecture, showing that this number is at most $C n^{3d/4}$ for $k<d$. As a consequence, we obtain an upper bound of $C n^{3d/4+2}$ for the number of similar $k$-simplices determined by $n$ points in $\R^d$, which improves the results of Agarwal, Apfelbaum, Purdy and Sharir. This problem is motivated by the problem of exact pattern matching.

We also address Zarankiewicz-type questions of finding the maximum number of edges in  semi-algebraic graphs with no  $K_{u,u}$. Here, we  improve the previous result of Fox, Pach, Sheffer, Suk, and Zahl, and Do for $d\le 4$, as well as for any $d$ and moderately large $u$. We get an improvement of their results for any $d$ and $u$ for unit-distance graphs, which was one of the main applications of their results.

From a more general prospective, our results are proved using classical cutting techniques. In the recent years, we saw a great  development of the polynomial partitioning method in incidence geometry that followed  the breakthrough result by Guth and Katz. One consequence of that development is that the attention of the researchers in incidence geometry swayed in polynomial techniques. In this paper, we argue that there is a number of open problems where classical techniques work better.
\end{abstract}

\newpage

\section{Intoduction}

In this paper, we study problems that are related to the following famous and fruitful Unit Distance Problem of Paul Erd\H os \cite{Erdosunit}: given a set of $n$ points on the plane, what is the largest number of pairs of them that are at unit distance apart? It is still wide open, with the best upper bound $O\left (n^{4/3}\right )$ due to Spencer, Szemer\' edi and Trotter~\cite{Spencerszemereditrotter} and the best lower bound $\mathrm{C}(2,2,n)=n^{1+\Omega(1/\log\log n)}$ due to Erd\H os himself \cite{Erdosunit}. A closely related question on the number of distinct distances, however, was almost resolved several years ago by Guth and Katz \cite{GK} using advanced algebraic techniques. The Unit Distance Problem gets somewhat simpler in $\R^3$, where the current best upper bound is $\mathrm{C}(3,2,n)= O\left (n^{295/197+\varepsilon}\right )$ due to Zahl \cite{zahl}, which is a recent improvement on the previous long-standing best upper bound of the form $O\left (n^{3/2+o(1)}\beta (n)\right )$ first due to Clarkson et al. \cite{clarkson} and then, in a slightly improved form, due to Kaplan et al. \cite{Kaplan} and Zahl \cite{unit3dzahl}. The lower bound is of the form $\Omega\left ( n^{4/3} \log \log n \right)$, see e.g. \cite{PA}.

In $\R^d$ with $d\ge 4$, the Unit Distance Problem changes its character and gets much easier since it is possible to have two sets $A,B$ of arbitrary sizes with any point from $A$ being at unit distance from $B$. (We describe this so-called Lenz construction below.) However, the problem has a natural generalisation that is non-trivial for higher dimensions. Let us denote by $\textrm{C}(d,k,n)$ the maximum, over all $k$-vertex simplices $S$ and $n$-element sets $P$ in $\R^d$, of the number of occurrences of a congruent copy of $S$ in $P.$   The following conjecture was made by Erd\H os and Purdy and Agarwal and Sharir.
\begin{conjecture}[Erd\H os and Purdy \cite{ErdosPurdy}, Agarwal and Sharir \cite{AS}]\label{conj unit simplices} For $d\geq 4$ and $k\leq d$ we have $\mathrm{C}(d,k,n)=O_d(\min\{n^k,n^{d/2}\})$ if $d$ is even, and $\mathrm{C}(d,k,n)=O_d(\min\{n^k,n^{d/2-1/6}\})$ if $d$ is odd.
\end{conjecture}
The conjecture is nontrivial only in the case $k>d/2$. Both bounds, if true, are tight.  Indeed, we can construct sets $P$ with $\Omega_d(\min\{n^k,n^{d/2-1/6}\})$ many {\it unit simplices}, i.e., simplices in which any two vertices are at unit distance apart,\footnote{In general, it seems to be of little significance if we consider congruent simplices or unit simplices in this question.} by using a generalised Lenz construction, described below (see also \cite{PA}).

Fix $t: = \lfloor d/2\rfloor$ circles of radius $\frac{1}{\sqrt{2}}$ centered at the origin and lying in pairwise orthogonal planes, and place $\lfloor \frac{n}{t} \rfloor$ or $\lceil \frac{n}{t} \rceil$ points on each of them, so that each point  is at distance $1$  from some other point on the same circle.  As the distance between any two points of different circles is $1$, by choosing the vertices from different circles, we obtain $\Omega_d(n^k)$ unit $k$-vertex simplices for any $k\le t$.
For even $d$ and $t<k\le d$, we can select either $1$ point or a pair of points at distance $1$ from each circle, which gives a set with  $\Omega_d(n^{d/2})$ $k$-simplices. For odd $d$ and $t<k\le d$, we need to replace one of the circles  by a $2$-sphere (note that we have an unused dimension for that). Then we can construct $\Omega_d\left (n^{d/2-1/6}\right )$ unit $k$-simplices in a similar way, using an arrangement of $\lfloor \frac{n}{t} \rfloor$ points on this $2$-sphere that determines  $\Omega_d(n^{4/3})$ unit distances (see \cite{sphere}).

This question is partly motivated by the group of algorithmic {\it pattern matching} problems:  given a set $P$ of $n$ points in $\R^d$ and a pattern $Q$, determine
whether $P$ contains a congruent or similar copy of $Q$, and, in some cases, find all of these copies. A standard approach to this problem  is to pick a spanning simplex $\Delta$ in $Q$, find all copies of $\Delta$ in $P$, and then test for each of them whether it extends to $Q$. The running time of such algorithms depends, first, on the time that is needed to find all the occurrences of $\Delta$ and, second, on the number of occurrences of $\Delta$ in $P$. We note that the first part could be done efficiently using cuttings (see \cite{AS} for details). For more about these algorithms and the connection between combinatorial geometry problems and pattern recognition, we refer the reader to \cite{Bras}.

Improving upon several bounds that were known for some specific instances of the problem, Agarwal and Sharir~\cite{AS} confirmed Conjecture~\ref{conj unit simplices} up to an arbitrarily small additive error of $\varepsilon$ in the exponent for $d=4,6$ and $k\le d-1$, as well as for $k=3$ and $d=5$.

One of the main results of this paper is as follows.

\begin{theorem}\label{unitsimplices} For any $d\ge 4$, $k\leq d+1$ and $\varepsilon>0$, we have  $\mathrm{C}(d,k,n)=O_{d,\varepsilon}\left ( n^{\frac{5}{8}d+\frac{1}{8}k+\varepsilon}\right )$.
\end{theorem}
In particular, this result implies $\mathrm{C}(d,k,n)=O_{d,\varepsilon}\left ( n^{\frac{3}{4}d+\varepsilon}\right)$ for $k\le d$. The first part of the proof follows the approach of Agarwal and Sharir \cite{AS}, however, in the second part we make better use of the geometric and combinatorial information that we possess and progress further. 

As a corollary, we obtain improved upper bounds on the maximum possible number of similar simplices. Let $\mathrm{S}(d,k,n)$ be the maximum number of simplices  spanned by a set of $n$ points in $\mathbb{R}^d$ that are similar to a given $k$-vertex simplex. Since a set of $n$ points determines at most $\binom{n}{2}$ distances, we have $\mathrm{S}(d,k,n)\leq n^2
\cdot \mathrm{C}(d,k,n)$, and thus in general we obtain $\mathrm{S}(d,k,n)=O_{d,\varepsilon}\left ( n^{\frac{3}{4}d+2+\varepsilon}\right )$ as a corollary of Theorem \ref{unitsimplices}. This significantly improves upon the previously  best known upper bounds for $k=d$ and $k=d+1$, which are $\mathrm{S}(d,d,n)=O_d(n^{d-8/5})$ and $\mathrm{S}(d,d+1,n)=O_d^*(n^{d-72/55})$, due to Agarwal, Apfelbaum, Purdy and Sharir \cite{similar}.

While the behavior of $C(d,k,n)$ does not seem to change much if we replace general congruent simplices with unit simplices in the definition, there is another natural setting in which the answer is significantly different. Let us define  $\mathrm{DIAM}(d,k,n)$ to be the maximum number of unit $k$-simplices spanned by a set of $n$ points {\it of diameter $1$}. Swanepoel determined  the exact value of this quantity for $k =2$ \cite{Swanepoel}. It was a conjecture of Schur \cite{schur} that $\mathrm{DIAM}(d,d,n) = n$, proved by Polyanskii and the second author \cite{proofschur}. In \cite{schur} it was also shown that $\mathrm{DIAM}(d,d+1,n)=1$. Here we confirm, up to a small error term, Conjecture~\ref{conj unit simplices} in this `diameter graph' settting. 

\begin{theorem}\label{diam simplices} For any $2\leq k \leq d+1$ and $\varepsilon>0$ we have $\mathrm{DIAM}(d,k,n)=O_{d,\varepsilon}\left (n^{\frac{d}{2}+\varepsilon}\right )$.
\end{theorem}

We, however, believe that something stronger should hold for such simplices.

\begin{conjecture}\label{conj diam} For $d\geq 2$ and $k\leq d$ we have $\mathrm{DIAM}(d,k,n)=O_d(\min\{n^k,n^{d-k+1}\})$ and, for odd $d$ and $k=(d+1)/2$, we have $\mathrm{DIAM}(d,k,n)=O_d(n^{k-1})$.
\end{conjecture}
The bounds, if true, are sharp, which can be shown using a modification of the generalised Lenz construction.\footnote{Unlike the general congruent simplices case, we can have only one diameter formed by the points on a given circle, which is the reason why the exponent decreases as $k>d/2$ increases.}
The known values of $\mathrm{DIAM}(d,k,n)$ are in accord with this conjecture.\\

Consider the following natural generalisation of the pattern matching problem: imagine that we search for an $m$-vertex pattern from a class $\mathcal Q$ that includes all patterns that have specified distances between {\it certain} pairs of vertices. (In the exact pattern matching, distances between {\it any} two pairs of points are specified.) We can encode such a pattern by a graph with vertices $\{1,\ldots, m\}$ and edges corresponding to the pairs of points that have a specified distance between them. A question of such type was recently considered by Palsson, Senger and Sheffer \cite{chains}. They asked for the maximum number of chains of length $k$ that can occur on a set of $n$ points in $\R^2$ or $\R^3$. In the above terms, a chain corresponds to a graph of a path. The results of \cite{chains} were sharpened to essentially optimal bounds by the authors of the present paper \cite{Paths}. There was a flurry of activity around this problem recently. Some of the variants included studying this question for trees and cycles. We refer the reader to \cite{Paths} for more details and references therein.

Note that the question for chains becomes trivial in $\R^d$ with $d\ge 4$ as we have Lenz-type constructions. We get a lower and upper bound of order $n^k$ for chains with $k$ vertices. In general, for any pattern $\Delta$ on $k$ vertices, one always has the trivial upper bound $n^k$ on the maximum possible number $\Delta(d,n)$ of congruent copies of $\Delta$  that can be determined by a set of $n$ points in $\mathbb{R}^d$. Below we give a characterisation of those sets $\Delta$ for which the trivial exponent $k$ can be improved. 
Let $G_{\Delta}$ be the graph on $k$ vertices, such that a pair of vertices form an edge if and only if the distance between the corresponding points in $\Delta$ is fixed.

We call a graph $G$ an \emph{orthogonality graph of planes in $\mathbb{R}^d$} if we can correspond to each vertex a  $2$-dimensional linear subspace in $\mathbb{R}^d$, such that a pair of vertices forms  an edge in $G$ if and only if the corresponding subspaces are orthogonal.

\begin{theorem}\label{nontrivi}
For a pattern $\Delta$ on $k\geq 1$ vertices we have $\Delta(d,n)=o_d(n)$ if and only if $G_{\Delta}$ is not a subgraph of an orthogonality graph in $\mathbb{R}^d$. Moreover, if $G_{\Delta}$ is not a subgraph of an orthogonality graph in $\mathbb{R}^d$, then
\[\Delta(d,n)=O_d(n^{k-\min \{1,4/d\}}).\]
\end{theorem}

Note that in the theorem above, the number of vertices $k$ of $\Delta$ is not necessarily bounded by $d$. We also note that orthogonality graphs have been studied for example by Erdős and Simonovits \cite{ES} and can be fairly complicated. In particular, it is not very difficult using a Frankl--Wilson type reasoning to show that the chromatic number of such graphs can be exponential in $d$.

\bigskip

\noindent \textbf{Zarankiewicz's problem for semi-algebraic graphs}\vskip+0.1cm
The classical Zarankiewicz's problem asks for the maximum number of edges in a bipartite graph with no $K_{s,t}$ as a subgraph, and is one of the central unresolved (and notoriously hard) problems in extremal combinatorics. The classical bound due to K\"ov\'ari--S\'os-Tur\' an states that such graphs have at most $O_t(n^{2-1/s})$ edges for $s<t$. There are few cases when the exponent in this bound is known to be optimal: it is the case for $t>(s-1)!$, which is due to Alon, R\' onyai  and Szab\'o \cite{ARS}, improving on an earlier construction due to Koll\'ar,  R\' onyai  and Szab\'o \cite{KRS}. Recently, Boris Bukh \cite{BB} managed to get the lower bound of the same form for $t>9^{(1+o(1))s}$. 

In geometric scenarios, however, the situation is sometimes much clearer, and we can hope for far better upper bounds. Let us give the relevant definition of a `geometric' graph. 

A graph $G=(V,E)$ is a {\it semi-algebraic graph in $\mathbb{R}^d$ of description complexity $t$} if $V\subseteq \mathbb{R}^d$ and there exist polynomials $f_1,\dots,f_t\in \mathbb{R}[x_1,\dots,x_{2d}]$ of degree at most $t$ and a Boolean function $\Phi(X_1,\dots,X_t)$ such that \[(p,q)\in E \Leftrightarrow \Phi\left (f_1(p,q)\geq 0, \dots, f_t(p,q)\geq 0\right )=1.\]

Semi-algebraic graphs arise naturally in computational geometry and its applications to range searching, motion planning etc. (see, e.g., \cite{AMat, AMS}).
Combinatorial properties of semi-algebraic graphs have been studied in the past few years, see for example \cite{repcomp, Conlon, Fox, density, regularity}. 

For semi-algebraic graphs, forbidding $K_{u,u}$, one can obtain upper bounds that  depend much more on  the dimension $d$ of the ambient space than on $u$ itself.\footnote{We note that forbidding $K_{u,u}$ in semi-algebraic graphs may be seen as a common abstraction of different general position requirements on the underlying set of points. E.g., if points in $\R^d$ are in general position then there is a unique hyperplane passing through any $d$ points, which means that the graph of incidences of points and hyperplanes is $K_{d,2}$-free.
} Although this problem was studied for particular classes of semi-algebraic graphs before, and especially for point-hyperplane incidence graphs \cite{Aps, BK,Chaz2} and intersection graphs \cite{Mustafapach}, it was suggested in full generality by Fox, Pach, Sheffer, Suk and Zahl~\cite{Fox}. The main result of \cite{Fox} is as follows.
\begin{thm}[\cite{Fox}] Let $G=(P_1,P_2,E)$ be a semi-algebraic bipartite graph in $\mathbb{R}^d$ with $|P_1|=m$, $|P_2|=n$ of description complexity $t$. If $G$ is $K_{u,u}$-free, then for any $\varepsilon>0$
\[|E(G)|=O_{d,t,u,\varepsilon}\left ((mn)^{\frac{d}{d+1}+\varepsilon}+m+n\right ).\]
\end{thm}
This was generalized by Do~\cite{repcomp} to the case of growing $u$, who showed the following.
\begin{thm}[\cite{repcomp}]\label{semialg Fox} Let $G=(P_1,P_2,E)$ be a semi-algebraic bipartite graph in $\mathbb{R}^d$ with $|P_1|=m$, $|P_2|=n$ of description complexity $t$. If $G$ is $K_{u,u}$-free, then for any $\varepsilon>0$
\[|E(G)|=O_{d,t,\varepsilon}\left (u (mn)^{\frac{d}{d+1}+\varepsilon}+um^{1+\varepsilon}+un^{1+\varepsilon}\right ).\]
\end{thm}

The results above were proved using polynomial partitioning techniques. Using classical cuttings, we improve Do's bounds for $d=2,3,4$ and all $u$, as well as for $d> 4$ if $u$ is sufficiently large.\footnote{More precisely, if $u\geq n^{\frac{2d-8}{(d+1)(2d-5)}}$.}
\begin{theorem}\label{zarenkiewicz semialg} Let $G=(P_1,P_2,E)$ be a semi-algebraic $K_{u,u}$-free bipartite graph in $\mathbb{R}^d$ with $|P_1|=m$, $|P_2|=n$, $1\leq u \leq n$ and of description complexity $t$. Then for any $\varepsilon>0$ the following hold.
\begin{enumerate}
\item If $d=2,3,4$, then $|E(G)|= O_{t,\varepsilon} \left ( u^{\frac{2}{d+1}}(mn)^{\frac{d}{d+1}+\varepsilon}+um^{1+\varepsilon}+un^{1+\varepsilon}\right )$.

\item If $d\geq 5$, then $|E(G)|= O_{d,t,\varepsilon}\left ( u^{\frac{2}{2d-3}} (mn)^{\frac{2d-4}{2d-3}+\varepsilon}+um^{1+\varepsilon}+un^{1+\varepsilon}\right )$.
\end{enumerate}
\end{theorem}

A heuristic argument in \cite{repcomp} shows that in the $d=2$ case the bound in Theorem~\ref{zarenkiewicz semialg} is best possible. We also note that our proof is arguably simpler than that of \cite{Fox} and \cite{repcomp}.

Note that, although all results in this section are stated for semi-algebraic bipartite graphs, they are also valid for any semi-algebraic graphs since any graph contains a bipartite subgraph with at least half the edges. We also note that recently Janzer and Pohoata \cite{Janzer} proved a Zarankiewicz type result for graphs of bounded VC-dimension.\\

\vskip+0.1cm

A {\it unit distance graph} in $\R^d$ is a semi-algebraic graph with vertices being points in $\R^d$ and pairs of points connected by an edge if and only if they are at distance $1$. A {\it diameter graph} is a unit distance graph with vertex set that has diameter $1$ (i.e., each edge corresponds to a diameter of the set of vertices). In the first part of the introduction, we were essentially dealing with the maximum number of cliques of fixed size in $n$-vertex unit distance or diameter graphs in $\R^d$. For unit distance graphs, we obtain a result that strengthens the bound of Do \cite{repcomp} (and Fox, Pach, Sheffer, Suk and Zahl \cite{Fox}) for all $u$ and $d$.

\begin{theorem}\label{zarankiewicz unit distances}
Let $G=(P_1,P_2,E)$ be a bipartite unit distance graph in $\mathbb{R}^d$ with $|P_1|=m$, $|P_2|=n$. If $G$ is $K_{u,u}$-free, then for any $\varepsilon>0$  \[ |E(G)|=O_{d,\varepsilon}\left ( u^{\frac{2}{d+1}}(mn)^{\frac{d}{d+1}+\varepsilon}+um^{1+\varepsilon}+un^{1+\varepsilon} \right ).\]
\end{theorem}
We  note that the case of unit distance graphs was pointed as one of the main applications of the general result in \cite{Fox}.

What is the largest number of $k$-cliques in a graph with no complete bipartite subgraph? For $k=2$ it is the Zarankiewitz problem. For larger $k$ this question was studied by Alon and Shikhelman \cite{Ash}, partly motivated by the discussions of unit distance graphs with the second author of the present paper. The bound, as in the Zarankiewitz case, of course heavily depends on the size of the forbidden complete bipartite graph. Below, we get an analogous extension of our results for unit distance graphs. For simplicity, we only include a result for the most difficult, $k=d+1$ case.

\begin{theorem}\label{simplex no K_{t,t}}
Let $\varepsilon>0$ and assume that the unit distance graph on a set of $n$ points $P\subseteq \mathbb{R}^d$ is $K_{t,t}$-free. Then $P$ determines $O_{d,t,\varepsilon}(n^{{(d+2)}/3+\varepsilon})$ many unit simplices with $d+1$ vertices for.
\end{theorem}
Let us make a remark about this result. It gives a better bound than what we are able to extract from the Agarwal--Sharir framework for this case (which would be roughly $n^{d/2}$). It is also better than the general bound that Do \cite{repcomp} gives for the analogous, but much more general, problem for semialgebraic hypergraphs. The approach for the proof is based on the ideas of our paper \cite{Paths} and, as we think, have potential to help in the full resolution of the Erd\H os--Purdy--Agarwal--Sharir problem.
\vskip+0.1cm

\textbf{Classical cuttings and polynomial partitioning}\vskip+0.1cm

Cuttings is a powerful technique in discrete and computational geometry and is the basis of geometric divide and conquer algorithms, with many applications to such algorithmic problems as range searching \cite{Ager, AMat, AMS, Chaz1, Mat2}, computing convex hulls and Voronoi diagrams \cite{Bron}  etc. (see the book \cite{HarPel}).

The idea of using sparsely intersecting simplices for divide and conquer goes back to Clarkskon \cite{Clark1} and Haussler and Welzl \cite{epsilon}, while the term "$r$-cutting" was coined by Matou\v sek \cite{Mat1}. In the years that followed, the cutting technique developed rapidly, through the work of
Agarwal, Aronov, Chazelle, Clarkson, Edelsbrunner, Friedman, Matousek, Sharir and many others. One of the key ingredients is the use of $\varepsilon$-nets (see the recent survey \cite{MV16}).

In the seminal paper \cite{clarkson}, Clarkson et al. introduced cuttings to incidence geometry. Since then, many Szemer\'edi--Trotter type theorems, as well as results on the number distinct or same-type configurations in the arrangements of points and surfaces were proved using these techniques (see \cite{AS2, matousek} and the references therein).

Recently, we saw a great  development of the polynomial partitioning method (see, e.g., \cite{Kaplan, PZ, Raz, zahl, unit3dzahl}) that followed  the breakthrough result by Guth and Katz \cite{GK}. One consequence of that development is that the attention of the researchers in incidence geometry swayed in polynomial techniques. For classical cuttings, the partition is obtained using a sample of the objects in question, and the arguments have probabilistic nature. For polynomial partitioning, the space is partitioned using the zero set of a high-degree polynomial, and the main tools for the study are of algebraic nature. In this paper, we make use of classical cutting techniques, showing that the approach can beat more algebraic methods.\\

\textbf{Structure}\vskip+0.1cm
The structure of the remaining part of the paper is as follows.

\begin{itemize}
\item In Section~\ref{sec11} we present several more results of lesser importance, including the resolution of some particular instances of the Erd\H os-Purdy conjecture and a Zarankiewicz result for diameter graphs. 
    \item In Section~\ref{sec3}, we overview the cutting results that we use. At the end of the section, we give a high-level overview of the proofs.
    \item In Section~\ref{sec4}, we prove the Zarankiewicz-type results using cutting results from Section~\ref{sec3}. 
    \item In Section~\ref{sec5}, we present the reduction of the problem of determining $\mathrm{C}(n,k,d)$ to a certain family of LP problems. This section follows the paper \cite{AS} and is presented for completeness. For simplicity, we describe it for unit simplices. 
    \item In Section~\ref{sec6}, we analyse the aforementioned LP's and prove Theorem~\ref{unitsimplices}.
    \item In Section~\ref{sec8}, we sketch the proof of Theorem~\ref{diam simplices}.
    \item In Section~\ref{sec10}, we prove Theorem~\ref{nontrivi}.
    \item In Section \ref{no K_tt} we prove Theorem \ref{simplex no K_{t,t}}
    \item In Section~\ref{sec7}, we prove the results about small simplices.
    \item In Section \ref{sec LP} we prove further results about the LP's that could be helpful for resolving Conjecture \ref{conj unit simplices} completely.
    \item in Appendix, we present an exposition of the proofs of the cutting results we use. We thought it could be useful because different necessary pieces are scattered around the literature.
    \end{itemize}

\subsection{Further results}\label{sec11}

In this subsection, we collected several results of lesser importance that we can prove using our methods.

We confirm (up to an $\varepsilon$ in the exponent) the smallest non-trivial instance of Conjecture~\ref{conj unit simplices} for general $d$: the case $k=\lfloor d/2\rfloor +1.$

\begin{theorem}\label{smallk}
Let $d>2$ be even, $k=\frac{d}{2}+1$ and $\varepsilon>0$. Then $\mathrm{C}(d,k,n)=O_{d,\varepsilon}\left (n^{\frac{d}{2}+\varepsilon}\right )$.
Let $d>3$ be odd, $k=\frac{d+1}{2}$ and $\varepsilon>0$. Then $\mathrm{C}(d,k,n)=O_{d,\varepsilon}\left (n^{d/2-\frac{1}{6}+\varepsilon}\right )$.
\end{theorem}

In the case when $d=7$ and $k=4,$ we managed to remove $\varepsilon$ error term from the exponent. Note that this resolves (up to the order of magnitude) the smallest case left open by Agarwal and Sharir \cite{AS}.

\begin{theorem}\label{k=4,d=7}
We have  $C(7,4,n)=\Theta(n^{10/3})$.
\end{theorem}

We also confirm the first open case of Conjecture \ref{conj diam} on simplices in diameter graphs.

\begin{theorem}\label{diam-triangle}
We have $\mathrm{DIAM}(5,3,n)=\Theta(n^{2})$.
\end{theorem}

Finally, we prove that for diameter graphs, we can replace $d$ by $\lfloor (d+1)/2\rfloor$ in the exponents of the bound in Theorem~\ref{zarankiewicz unit distances}. 
\begin{theorem}\label{zarankiewicz diameters}
Let $G=(P_1,P_2,E)$ be a bipartite diameter graph in $\mathbb{R}^d$ with $|P_1|=m$, $|P_2|=n$. If $G$ is $K_{u,u}$-free, then for any $\varepsilon>0$  \[|E(G)|=O_{d,\varepsilon} \left (u^{\frac{2}{\lfloor (d+1)/2 \rfloor+1}}(mn)^{\frac{\lfloor (d+1)/2\rfloor}{\lfloor (d+1)/2\rfloor+1}+\varepsilon}+um^{1+\varepsilon}+un^{1+\varepsilon}\right ).\]
\end{theorem}
In particular, if we have an $n$-vertex diameter graph in $\R^4$ without a $K_{100,100}$, then the number of edges in such a graph is at most $C_{\epsilon}n^{4/3+\epsilon}$.

\section{Cuttings}\label{sec3}

Let $\Xi$ be a decomposition (or sometimes called a subdivision) of a set $U\subseteq \mathbb{R}^d$ into relatively open cells. The \emph{size} of $\Xi$ is its total number of cells.

\begin{definition}
An algebraic surface $\sigma$ \emph{crosses} a cell $\tau$ of $\Xi$, if $\emptyset\ne \sigma$ and $\sigma\cap \tau\ne C$. In words, $\sigma$ crosses $\tau$ if $\sigma$ contains some, but not all, points of $C$.
\end{definition}

Let $\Sigma=\{\sigma_1,\dots,\sigma_m\}$ be a collection of algebraic surfaces. $\Xi$ is called \emph{$1/r$-cutting} of $\Sigma$, if each cell of $\Xi$ is crossed by at most $|\Sigma|/r$ members of $\Sigma$.
The cutting results below guarantee the existence of small $1/r$-cuttings for collections of surfaces. We will need a variety of cutting results for different questions that we address. Although most of the cutting results we use appeared in the literature before in some form, in some cases they are rather implicit. Since extracting the explicit statements and their proofs are far from being obvious, we decided to sketch their proofs in the Appendix. In this section, we only state the results.

It is important to emphasise that \emph{crossing} is a different notion than \emph{intersecting}. The difference is  crucial. The cutting results we use are not true with if crossing is replaced with intersection.

In the proof of Theorem~\ref{zarenkiewicz semialg} on semi-algebraic graphs, we will use the following results of Chazelle et al. and of Koltun \cite{curvecut,Koltun}.

\begin{thm}\label{semialg cutting} Let $\Sigma$ be a collection of $n$ algebraic surfaces of degree at most $t$ in $\mathbb{R}^d$. Then the following is true for any $\varepsilon>0$ and $1\leq r \leq n$.
\begin{enumerate}
\item[(i)] \cite{curvecut,Koltun} If $d=2,3,4$ then there is a subdivision $\Xi$ of $\mathbb{R}^d$ of size $O_{t,\eps}\left ( r^{d+\eps}\right)$ which is a $1/r$-cutting $\Xi$ of $\Sigma$. Moreover, for any set of $m$ points $P\subseteq\mathbb{R}^d$ the subdivision $\Xi$ can be chosen so that each cell contains at most $m/r^d$ points of $P$.

\item[(ii)] \cite{Koltun} If $d\geq 5$ then there is a subdivision $\Xi$ of $\mathbb{R}^d$ of size $O_{d,t,\varepsilon}\left (r^{2d-4+\varepsilon}\right )$ which is a $1/r$-cutting of $\Sigma$ of size $O_{d,t,\varepsilon}\left (r^{2d-4+\varepsilon}\right )$. Moreover, for any set of $m$ points $P\subseteq\mathbb{R}^d$ the subdivision $\Xi$ can be chosen so that each cell contains at most $m/r^{2d-4}$ points of $P$.
    \end{enumerate}
    \end{thm}

In the proof of Theorems~\ref{unitsimplices} and~\ref{zarankiewicz unit distances}, we will use cuttings for collections of spheres. The \emph{dimension} of a sphere $\mathbb{S}^{d-1}$ in $\mathbb{R}^d$ is $(d-1)$. The proof of the following cutting result is sketched by Agarwal and Sharir \cite{AS}.
    
    \begin{thm}[\cite{AS}]\label{1/r-cutting} Let $\Sigma_1,\dots,\Sigma_k$ be $k$ sets of spheres in $\mathbb{R}^d$. Then the following statement hold for every $1\leq r \leq n$ and every $\varepsilon>0$.
    \begin{enumerate}
    \item There is a subdivision $\Xi$ of $\mathbb{R}^d$ of size $O_{d,k,\varepsilon}\left (r^{d+\varepsilon}\right )$ which is a $1/r$-cutting of each $\Sigma_i$. Moreover, for any set of $m$ points $P\subseteq \mathbb{R}^d$ the subdivision $\Xi$ can be chose so that each cell contains at most $m/r^d$ points of $P$.
    \item If $\mathbb{S}$ is sphere of dimension $\lambda$ in $\mathbb{R}^d$, then there is  a subdivision of $\mathbb S$ of size $O_{d,k,\varepsilon}\left (r^{\lambda+\varepsilon}\right )$, which is a $1/r$-cutting of each of $\Sigma_i$. Moreover, for any set of $m$ points $P\subseteq \mathbb{S}$ the subdivision $\Xi$ can be chose so that each cell contains at most $m/r^d$ points of $P$. 
    \end{enumerate}
    \end{thm}

In the proofs of Theorems~\ref{diam simplices} and~\ref{diam-triangle} we use the following cutting result. 
    
\begin{theorem}\label{diam cutting} Let $\Sigma_1,\dots,\Sigma_k$ be $k$ sets of spheres in $\mathbb{R}^d$ and $B$ be the intersection of balls bounded by the members of $\bigcup_{i\in[k]} \Sigma_i$. Then the following hold for every $1\leq r \leq n$ and $\varepsilon>0$.
\begin{enumerate}
\item There is a subdivision $\Xi$ of $B$ of size $O_{d,k,\varepsilon}\left (r^{\lfloor (d+1)/2 \rfloor+\varepsilon}\right )$, which is a $1/r$-cutting of each $\Sigma_i$. Moreover, for any set of $m$ points $P\subseteq B$ the subdivision $\Xi$ can be chosen so that each cell contains at most $m/r^{\lfloor (d+1)/2 \rfloor}$ points of $P$.
    
\item If $\mathbb{S}$ is a sphere of dimension $\lambda$ in $\mathbb{R}^d$, then there is subdivision $\Xi$ of the set $\mathbb S\cap B$ of size $O_{d,k,\varepsilon}\left (r^{\lfloor (\lambda+1)/2 \rfloor+\varepsilon}\right )$, which is a $1/r$-cutting of each  $\Sigma_i$. Moreover, for any set of $m$ points $P\subseteq B$ the subdivision $\Xi$ can be chosen so that each cell contains at most $m/r^{\lfloor (\lambda+1)/2 \rfloor}$ points of $P$.
    \end{enumerate}
    \end{theorem}

\begin{remark} 
We included in all statements a bound on the number of points of a given set in each cell. It  is usually not part of similar theorems, and we include it for convenience. It is very easy to derive from just the first part of each statement.  If a cell contains $\beta$ times more points than what is guaranteed by the second part of each statement, then we use $\lfloor \beta\rfloor$ hyperplanes to subdivide it into $\lfloor \beta\rfloor+1$ cells each of which satisfies the requirement. Note that the increase in the number of cells is insignificant and can be hidden in the $O$-notation.
\end{remark}

We mention that the main difference between the proofs of Theorem \ref{diam cutting} and  Theorem \ref{1/r-cutting} is that in the former we observe that, instead of using that $r$ hyperplanes divides $\mathbb{R}^d$ into $O_d(r^d)$ cells, we can use the fact that a $d$-dimensional polytope with $r$ facets has $O_d(r^{\lfloor d/2 \rfloor})$ faces. An exact bound is given by the Upper Bound Theorem \cite{upperbound}. 

\subsection{Overview of the proofs}
In the proofs of the results, we employ a divide-and-conquer procedure. The  points in question are partitioned into groups (in some cases, such as the bipartite Zarankiewicz's problem, the two groups are the classes of the bipartition), and on each step one group is treated as points, while the other groups are treated as surfaces. The overall problem reduces to an incidence-type problem. We then partition the space into cells using cutting results, and apply the same procedure on each of the cells.

One important feature of the proofs, which appears when we partition the spaces into cells, is the dichotomy between surfaces crossing a cell and surfaces containing a cell. This can be put in parallel with the structure (containing) and quasirandomness (crossing) dichotomy. Cutting results only guarantee certain properties for surfaces that cross cells, and thus surfaces that cross cells are amenable to random techniques. Nothing can be inferred about surfaces containing a cell using cuttings, however, in this situation we gain a lot of structural information. For example, in the case when we work with unit distance graphs or unit simplices,  unit spheres centred at the points play the role of surfaces. Then, having a group of unit spheres that contain a cell means that the centres of these spheres form a complete bipartite graph with the points that lie inside the cell. 

In many cases, such as the Zarankiewicz-type results, we can argue that such situations are impossible. In other questions, such as the congruent simplices question, if the divide and conquer does not progress, we gradually make the situation more and more structured, until it is easy to deal with. 

Our proof of the results on the number of simplices  follow the proof of Agarwal and Sharir~\cite{AS}.  Using  their technique, they managed to obtain a  recurrence for $\mathrm{C}(d,k,n)$. Solving this recurrence leads to the bound $\mathrm{C}(d,k,n)=O(n^{\zeta(d,k)+\eps})$, where $\zeta(d,k)$ is a feasible solution for a family of LP problems, indexed by certain weighted graphs (see Section~\ref{sec6} for details). They solved these LP's for many $k$ and $d\le 7$. However, they did not manage to progress on solving these LP's for general $d$ and left it as an open problem. The main new part of our proofs here is the analysis of these LP's that uses some additional geometric information and combinatorial ideas.

\section{Proof of Zarankiewicz-type results}\label{sec4}
In what follows, for a graph $G = (V,E)$ and two disjoint subsets of the vertex set $X,Y\subset V$, we denote by $E(X,Y)$ the set of all edges of $G$ that connect a vertex from $X$ with a vertex from $Y$. We also denote by $G(X,Y)$ the bipartite graph with vertex set $X\cup Y$ and edge set $E(X,Y).$

First we illustrate the common structure of the proofs on the example of Theorem~\ref{zarankiewicz unit distances} with $|A|=|B|$ and for constant~$u$.

\begin{proof}[Proof of Theorem~\ref{zarankiewicz unit distances} with $|P_1|=|P_2|$ and for constant $u$] 

Let $g(m,n)$ be the maximum number of unit distances between two disjoint sets of size $m$ and $n$ in $\R^d$, given that the induced unit distance graph between the two sets does not contain $K_{u,u}$. 
In these terms, Theorem~\ref{zarankiewicz unit distances} in our particular case states that $g(n,n)=O_{d,\varepsilon}(n^{2d/(d+1)+\varepsilon})$. First we derive a recursive inequality for the asymmetric $g(m,n)$,  and from the inequality we derive the statement in the symmetric case.

Take point sets $P_1,P_2\subseteq \mathbb{R}^d$ with $|P_1|=m$ and $|P_2|=n$, such that the number of unit distances between $P_1$ and $P_2$ is equal to $g(m,n)$. Let $G = (P_1,P_2,E)$ be the unit distance graph between $P_1$ and $P_2.$

For each $p\in \R^d$, let $\sigma_p$ be the unit sphere centred at $p$. Put $\Sigma:=\{\sigma_p: p\in P_1\}$  and let $r>0$ be a parameter to be chosen later. Note that two points $p_1\in P_1$ and $p_2\in P_2$ are at unit distance apart if and only if $p_2\in \sigma_{p_1}$.

By Theorem~\ref{1/r-cutting}, there is a decomposition $\Xi$ of $\mathbb{R}^d$ into $O_{d,\varepsilon}(r^{d+\varepsilon})$ cells, such that each cell is crossed by at most $n/r$ spheres from $\Sigma$ and each cell contains at most $m/r^d$ points of $P_2$. 

For a cell $\tau\in \Xi$, put $P_1^{\tau}=P\cap \tau$ and let $\widetilde{P}_2^{\tau},  \widehat{P}_2^{\tau}\subseteq P_2$ be the sets of those $p\in P_2$ for which $\sigma_p$ crosses or contains $\tau$, respectively. (Note that $\widetilde{P}_2^{\tau}$ is disjoint from $\widehat{P}_2^{\tau}$.) A pair of points $p_1\in P_1^{\tau}$ and $p_1\in P_2$ is at unit distance apart only if $\sigma_{p_2}$ either crosses or contains $\tau$. Thus, each edge in $G$ belongs to $E(P_1^{\tau},\widetilde{P}_2^\tau\sqcup \widehat{P}_2^{\tau})$ for some $\tau\in \Xi$.  We obtain
\[E(P_1,P_2)=\bigcup_{\tau\in \Xi} \Big (E(P_1^{\tau},\widetilde{P}_2^{\tau})\sqcup E(P_1^{\tau},\widehat{P}_2^{\tau})\Big ).\]

Each cell $\tau$ is crossed by at most $n/r$ spheres and contains at most $m/r^d$ points of $P$. Moreover, the induced graph of unit distances is $K_{u,u}$-free. Thus, we obtain $|E(P_1^{\tau},\widetilde{P}_2^{\tau})|\leq g(m/r^d,n/r)$. Further, the graph of unit distances between $P_1^{\tau}$ and $\widehat{P}_2^{\tau}$ is complete bipartite, and thus one of the parts must have size smaller than $u$ since $G$ is $K_{u,u}$-free. Therefore, $|E(P_1^{\tau},\widehat{P}_2^{\tau})|\leq (u-1)(n+m)=O_u(n+m)$. These two inequalities imply
\[g(m,n)\leq O_{d,\varepsilon,u}(r^{d+\varepsilon})\left (g(m/r^d,n/r)+n+m\right ).\]
Applying the  bound above with $m' = n/r$ and $n' = m/r^d$, we get
\begin{equation*}
g(n/r, m/r^d)\leq O_{d,\varepsilon,u}(r^{d+\varepsilon})\left (g(n/r^{d+1},m/r^{d+1})+n/r+m/r^d\right ),
\end{equation*}
Now we use the fact that $f$ is symmetric in its two variables, namely that $g(n/r,m/r^d) = g(m/r^d,n/r)$, and combine the two displayed bounds:
\[g(m,n)\leq O_{d,\varepsilon,u}(r^{2d+2\varepsilon})\left (g(n/r^{d+1},m/r^{d+1})+n+m\right ).\]
Assuming that $r$ is a sufficiently large integer (that is, dependent on $d,\varepsilon$ only, but not on $n,m$) and that $n=m$, the inequality above implies $g(n,n)\le r^{2d+3\varepsilon}g(n/r^{d+1},n/r^{d+1})$. If $n$ has the form $n = r^{k(d+1)}$ for integer $k$ then the previous inequality, applied repeatedly, implies $g(n,n)\le r^{k(2d+3\varepsilon)}g(1,1) = n^{(2d+3\varepsilon)/(d+1)}$, and therefore for any sufficiently large $n$ we have $g(n,n)\le n^{(2d+4\varepsilon)/(d+1)}\le n^{2d/(d+1)+4\varepsilon}.$ Since $\varepsilon>0$ was arbitrary, we may replace $4\varepsilon$ in the last expression by $\varepsilon$. This immediately implies that $g(n,n)=O_{d,\varepsilon,u}(n^{2d/(d+1)+\varepsilon})$, which finishes the proof.
\end{proof}

The proofs of (all parts of) Theorem~\ref{zarenkiewicz semialg} and Theorem~\ref{zarankiewicz unit distances} all follow the same strategy and are identical modulo the cutting results we use. Thus, in what follows we only present the proof of the third part of Theorem~\ref{zarenkiewicz semialg}.

A collection $f_1,\dots,f_n\in \mathbb{R}[x_1,\dots,x_d]$ of polynomials define a partition of $\mathbb{R}^d$ into relatively open regions in the following way. Let $\approx$ be an equivalence relation on $\mathbb{R}^d$ such that $x\approx y$ iff $\setbuilder{i}{f_i(x)\gtreqqless0}=\setbuilder{i}{f_i(y)\gtreqqless 0}$. The regions of the partition are the connected components of the equivalence classes.

\begin{proof}[Proof of Theorem~\ref{zarenkiewicz semialg} part 2]
Let $h_u(m,n)$ denote the maximum number of edges of a $K_{u,u}$-free semi-algebraic bipartite graph in $\mathbb{R}^d$ of complexity $t$ and with classes of size $m$ and $n$. Take any such graph $G=(P_1\cup P_2,E)$ with $|E| = h_u(m,n)$ and parts $P_1,P_2$. Assume also that $m\ge n$.

Suppose that $G$ is defined by the polynomials $f_1,\dots,f_t$ of degree at most $t$ and the Boolean function $\Phi(X_1,\dots,X_t)$: 
\[(p_1,p_2)\in E \Leftrightarrow \Phi\left (f_1(p_1,p_2)\geq 0,\dots,f_t(p_1,p_2)\right)\geq 0)=1.\] 

For each $i\in [t],$ and $p_2\in P_2$, define the following  surfaces, determined by a $d$-variate polynomial of degree at most $t$:
\[\sigma_{p_2,i}:=\setbuilder{x\in \mathbb{R}^d}{f_i(x,p_2)=0}.\]
Put 
\[\Sigma:=\setbuilder{\sigma_{p_2,i}}{i\in [t], p_2\in P_2}.\]
Note that $\Sigma_2$ is a collections of $tn$ algebraic surfaces in $\R^d$ of degree at most $t$.

Let $r>0$ be a parameter to be chosen later.
By Theorem~\ref{semialg cutting}~(ii), there is a $1/tr$-cutting $\Upxi$ of $\Sigma$ of size $O_{t,\varepsilon,d} \left ( (tr)^{2d-4+\varepsilon} \right )=O_{t,\varepsilon,d} \left ( r^{2d-4+\varepsilon} \right )$ such that each cell of $\Upxi$ contains at most $m/r^{2d-4}$ points of $P_1$.

For each $p_2\in P_2$, consider the partition of $\mathbb{R}^d$  defined by the surfaces $\setbuilder{\sigma_{p_2,i}}{i\in [t]}$. Then there is a subset $R$ of these regions such that for a vertex $p_1\in P_1,$ we have $(p_1,p_2)\in E$ if and only if $p_1$ belongs to one of the regions from $R$. Thus, for any vertex $p_2\in P_2$ and cell $\tau\in \Upxi,$ one of the following three possibilities hold. 
\begin{enumerate}
\item $\tau$ is crossed by at least one member of $\setbuilder{\sigma_{p_2,i}}{i\in [t]}$.
\item For every vertex $p_1\in P_1\cap \tau$  we have $(p_1,p_2)\in E$.
\item For every vertex $p_1\in A\cap \tau$  we have $(p_1,p_2)\notin E$.
\end{enumerate}

Given $\tau\in \Upxi,$ let $\widetilde{P}_2^{\tau}\subseteq P_2$ be the set of those vertices $p_2$ for which there is an $i\in [t]$ such that $\tau$ is crossed by $\sigma_{p_2,i}$. In other words, $\widetilde{P}_2^{\tau}$ is the set of those vertices for which $1.$ holds. Further, let $\widehat{P}_2^{\tau}\subseteq P_2$ be the set of those vertices, for which $2.$ holds, and let $P_1^{\tau}:=P_1\cap \tau$.
Then we have
\[E=\bigcup_{\tau\in\Upxi} \left ( E(P_1^{\tau},\widetilde{P}_2^{\tau})\bigcup E(P_1^{\tau},\widehat{P}_2^{\tau})\right ).
\]

Remember that each $\tau\in \Upxi$ contains at most $m/{r^{2d-4}}$ points of $P_1$, and is crossed by at most $tn/tr=n/r$ members of $\Sigma$. Thus \[|E(P_1^{\tau},\widetilde{P}_2^{\tau})|\leq h_u\left (\frac{m}{r^{2d-4}},\frac{n}{r}\right ).\]

Since $G(P_1^{\tau}, \widehat{P}_2^{\tau})$ is a complete bipartite graph and $G$ is $K_{u,u}$-free, we have
\[|E(P_1^{\tau},\widehat{P}_2^{\tau})|\leq um.\]
Summing the two displayed inequalities over all cells, we obtain
\begin{equation}\label{eqrecur}h_u(m,n)\leq \sum_{\tau\in\Upxi}\left ( |E(P_1^{\tau},\widetilde{P}_2^{\tau})|+|E(P_1^{\tau},\widehat{P}_2^{\tau})| \right )=
O_{d,t,\varepsilon}\left ( r^{2d-4+\varepsilon}\right)\left( h_u\left (\frac{m}{r^{2d-4}},\frac{n}{r}\right )+um\right).
\end{equation}
Choosing $r=r(\varepsilon,d,t)$ large enough, we may assume that $r^{\varepsilon}\ge 2$ and that \eqref{eqrecur} implies the following inequality.
\begin{equation}\label{eqrecur2} h_u(m,n)\leq r^{2d-4+2\varepsilon}\left( h_u\left (\frac{m}{r^{2d-4}},\frac{n}{r}\right )+um\right).\end{equation}

Let us first show that there exists a constant $C'=C'(d,t,\eps),$ such that the following holds: if \begin{equation}\label{eqconcond}\frac mu\ge \left(\frac nu\right)^{2d-4},\end{equation} then 
\begin{equation}\label{statpart1}h_u(m,n)\le C'um^{1+\eps}.\end{equation} 

Remark that if \eqref{eqconcond} holds for $m,n$, then it also holds for $m/r^{2d-4},n/r$. 
We prove the claim by induction on $m+n$. The inequality \eqref{statpart1} is clear if $n\le u$.
Combining the induction hypothesis and  \eqref{eqrecur2}, we get
\begin{equation*}h_u(m,n)\leq 
r^{2d-4+2\varepsilon}\left( C'u\left(\frac{m}{r^{2d-4}}\right)^{1+\eps}+um\right).
\end{equation*}
By choosing $C'$ large enough, we may assume that $m$ is large enough, and in particular $m^{\eps}>r^{4d}$. Thus, we get that
\begin{equation*}h_u(m,n)\leq 
r^{2d-4+2\varepsilon}\cdot 2C'u\left(\frac{m}{r^{2d-4}}\right)^{1+\eps}<2C'r^{-\eps}um^{1+\eps}\le C'um^{1+\eps}.
\end{equation*}
Here, the second inequality is due to  $2d-4+2\eps-(2d-4)(1+\eps)< -\eps$, valid for $d\ge 4$, and the last inequality is due to $r^{\eps}\ge 2.$ This proves \eqref{statpart1}.

Assume next that the opposite of \eqref{eqconcond} holds. Note that this is equivalent to 
\begin{equation}\label{eqconcond2} um< u^{\frac 2{2d-3}}(mn)^{\frac{2d-4}{2d-3}}.
\end{equation}
In this assumption, we shall show that there exists a constant $C=C(d,t,\eps),$ such that the following inequality holds.
\begin{equation*}\label{statpart2}h_u(m,n)\le Cu^{\frac 2{2d-3}}(mn)^{\frac{2d-4}{2d-3}+\eps}.\end{equation*} 
We remark that if \eqref{eqconcond2} holds for $m>n$, then it also holds for $m/r^{2d-4},n/r$ (as well as for $m/r, n/r^{2d-4}$, provided $m,n>C_1(r)u$). The proof is again by induction on $m+n$. The statement is true for $n\le u$, since by \eqref{eqconcond2} we have $h_u(m,u)\le mu<u^{\frac 2{2d-3}}(mn)^{\frac{2d-4}{2d-3}}.$ By making $C$ large enough, we may again assume that $m$ is large enough: $m^{\eps}\ge r^{4d}$.

Then, by induction\footnote{To apply the induction hypothesis, we implicitly use the fact that the function $h_u(m,n)$ is symmetric.} and using \eqref{eqconcond2} we have
\begin{multline*}h_u(m,n)\leq 
r^{2d-4+2\varepsilon}\left( Cu^{\frac 2{2d-3}}\left(\frac{mn}{r^{2d-3}}\right)^{\frac{2d-4}{2d-3}+\eps}+u^{\frac 2{2d-3}}(mn)^{\frac{2d-4}{2d-3}}\right)\\
\le r^{2d-4+2\varepsilon}\cdot 2Cu^{\frac 2{2d-3}}\left(\frac{mn}{r^{2d-3}}\right)^{\frac{2d-4}{2d-3}+\eps}\le
C u^{\frac 2{2d-3}}(mn)^{\frac{2d-4}{2d-3}+\eps},
\end{multline*}
where the transitions as in the previous case. This concludes the proof.\end{proof}

\section{Unit and diameter simplices}\label{sec5}
A \emph{unit simplex} is a regular simplex of edge length $1$. Our goal for this and the following section is to prove Theorem~\ref{unitsimplices} for {\it unit $k$-simplices}. The proof in the general case is identical, but involves more notation, and we stick to the unit distance case for clarity. 

This section is essentially an overview of Section~$5$ from \cite{AS}. Here, we reduce the problem to a certain linear optimization question. 

\subsection{Notation}

\subsubsection*{Unit distances}

For $k$ finite point sets  $P_1,\dots,P_k\subseteq \mathbb{R}^d$ with $1\leq k \leq d+1$, let
\[\Psi_{k,d}(P_1,\dots,P_k):=\setbuilder{(p_1,\dots,p_k)}{p_i\in P_i \textrm{ for each } i\in[k] \textrm{ and } (p_1,\dots,p_k) \textrm{ is a unit simplex} },
\]
\[ \psi_{k,d}(P_1,\dots,P_k)=|\Psi_{k,d}(P_1,\dots,P_k)|.\] Next,  define
\[\psi_{k,d}(n_1,\dots,n_k):=\max_{|P_1|\leq n_1,\dots,|P_k|\leq n_k}\psi_{k,d}(P_1,\dots,P_k).\]

With this notation, the quantity $\mathrm{U}(d,k,n)$, equal to the maximum number of unit $k$-simplices spanned by a set of $n$ points in $\mathbb{R}^d$, is upper-bounded by $\psi_{k,d}(n):=\psi_{k,d}(n,\dots,n)$. We are going to bound the latter quantity. 

Given a $k$-tuple of point sets $(P_1,\dots,P_k)$ and a graph $G=([k],E),$ we say that $G$ is {\it compatible} with the $k$-tuple, if $\{i,j\}\in E$ implies that $\|p-p'\|=1$ for every pair of points $p\in P_i$ and $p'\in P_j$. Consider a vector $\pmb{\lambda}=(\lambda_1,\dots,\lambda_k)\in \mathbb{N}^k,$  where $\lambda_i$ is the smallest $\ell$ such that there is a sphere\footnote{The dimension of a circle is $1$.} of dimension $\ell$ containing $P_i$ with the convention that $\lambda_i=d$ if there is no such sphere and with $\lambda_i=0$ if $|P_i|\le 3$. We call $\pmb\lambda$ the {\it profile} of $(P_1,\ldots, P_k).$ 

For a graph $G=([k],E)$ and $i,j\in [k]$, $i\ne j$, we denote 
\[G\cup \{i,j\}:=([k],E\cup\{\{i,j\}\}).\]
Note that, whenever we use this notation, we have $\{i,j\}\notin E,$ and thus $G\cup \{i,j\}$ has a strictly bigger edge set than $G$. Further, define
{\footnotesize\[\psi_{k,d}^{(G)}(n_1,\dots,n_k):=\max\setbuilder{\psi_{k,d}(P_1,\dots,P_k)}{|P_i|\leq n_1,\dots,|P_k|\leq n_k, \textrm{ and }  \\ G  \textrm{ is compatible with }(P_1,\dots,P_k)},
\]}
and $\psi_{k,d}^{(G)}(n_1,\dots,n_k)=0$ if no $k$-tuple $(P_1,\dots,P_k)$ is compatible with $G$. Put \[\psi_{k,d}^{(G)}(n):=\psi_{k,d}^{(G)}(n,\dots,n).\] When the value of $d$ is fixed or well understood from the context, we omit it from the subscript.
Note that $\psi_k^{(G)}(n_1,\dots,n_k)\leq \psi_k(n_1,\dots,n_k)$ for any $G$, and if $G=([k],\emptyset)$ then $\psi_k^{(G)}(n_1,\dots,n_k)=\psi_k(n_1,\dots,n_k)$.

\subsubsection*{Diameters}

We introduce analogue notation for diameters. Let \[\phi_{k,d}(n_1,\dots,n_k)=\max_{\substack{|P_1|\leq n_1,\dots,|P_k|\leq n_k\\ \diam{\cup{P_i}}=1}}\psi_{k,d}(P_1,\dots,P_k)\] and

{\small\[\phi_{k,d}^{(G)}(n_1,\dots,n_k)
:=\max_{\substack{|P_1|\leq n_1,\dots,|P_k|\leq n_k\\ \diam{\cup{P_i}}=1}}\setbuilder{\psi_{k,d}(P_1,\dots,P_k)}{ G  \textrm{ is compatible with }(P_1,\dots,P_k)},
\]}
with $\phi_{k,d}^{(G)}(n_1,\dots,n_k)=0$ if no $k$-tuple $(P_1,\dots,P_k)$ for which $\diam{\cup P_u}=1$ is compatible with $G$. Put \[\phi_{k,d}^{(G)}(n):=\phi_{k,d}^{(G)}(n,\dots,n).\] Again, when the value of $d$ is fixed or well understood from the context, we omit it from the subscript.
Note that $\phi_k^{(G)}(n_1,\dots,n_k)\leq \phi_k(n_1,\dots,n_k)$ for any $G$, and if $G=([k],\emptyset)$ then $\phi_k^{(G)}(n_1,\dots,n_k)=\phi_k(n_1,\dots,n_k)$.

\subsection{The recurrence for the unit distance case}\label{rec}

For a graph $G=([k],E)$, let $P_1,\dots,P_k\subseteq \mathbb{R}^d$, $|P_1| =\ldots = |P_k|=n$ be point sets such that $\psi_k^{(G)}(n)=\psi_k^{(G)}(P_1,\dots,P_k)$. Let $\pmb{\lambda}$ be the profile of $(P_1,\dots,P_k)$, and, for each $i\in [k]$, let $S_i$ be a sphere of dimension $\lambda_i$ containing $P_i$ (if $\lambda_i=d$ then put $S_i=\mathbb{R}^d$). Throughout this subsection, both $G$ and $\pmb \lambda$ are fixed. 

We obtain a bound on $\psi_k^{(G)}(P_1,\dots,P_k)$ in $k$ rounds. In the $i$-th round, $P_i$ is treated as a set of points, while the other $k-1$ sets are treated as sets of unit spheres centred at the corresponding points. 

For every $i\in [k]$ let 
\[V_i:=\setbuilder{j\in [k]\setminus\{i\}}{\{i,j\}\notin E}.\]
For a point $x$, we write $\sigma_x$ for the unit sphere with centre in $x$. Also, for any set of points $P$, we write
\[\Sigma(P):=\{\sigma_x: x\in P\}.\]
{\bf Round 1. }

Let $r_1$ be a parameter to be chosen later. By Theorem~\ref{1/r-cutting}, there is a subdivision $\Upxi$ of $S_1$ of size $O_{d,\varepsilon}(r_1^{\lambda_1+\varepsilon}),$  which is a $1/r_1$-cutting of each $\Sigma(P_j)$, $j=2,\ldots,k$ such that each cell contains at most $n/r_1^{\lambda_1}$ points of $P_1$. For a cell $\tau\in \Upxi$, let $P_1^{\tau}=P_1\cap \tau$. We have
\[\psi_k^{(G)}(P_1,\dots,P_k)=\sum_{\tau\in \Upxi}\psi_k^{(G)}(P_1^\tau,P_2,\dots,P_k).\]

For every $j\in V_1$, set
\[\widetilde P_j^{\tau}:=\{p\in P_j\ :\ \sigma_p\text{ crosses }\tau\},\]
\[\widehat P_j^\tau:=\{p\in P_j\ :\ \sigma_p\text{ contains }\tau\}.\]
Note that $\widetilde P_j^\tau$ is disjoint from $\widehat P_j^\tau$ for every $j\in V_1$. For $j\in V\setminus (V_1\cup \{1\})$, set $\widetilde P_j^{\tau} = \widehat P_j^{\tau} = P_j.$
 With this notation, we have
\begin{equation*}
\psi_k^{(G)}(P_1,\dots,P_k) \le \sum_{\tau\in \Upxi}\psi_k^{(G)}(P_1^\tau,\widetilde P_2^\tau,\widetilde P_3^\tau,\dots,\widetilde P_k^\tau)+\sum_{\tau\in \Upxi}\sum_{j\in V_1}\psi_k^{(G)}(P_1^\tau,P_2,\dots,P_{j-1},\widehat P_j^\tau,P_{j+1},\dots,P_k).
\end{equation*}

For each cell $\tau$ and $j\ne 1,$ the distance between the points of $P_1^{\tau}$ and $\widehat P_j^{\tau}$ is $1$. Thus, for any $j\in V_1\setminus\{1\}$, the $k$-tuple $(P_1^\tau,P_2,\dots,P_{j-1},\widehat P_j^\tau,P_{j+1},\dots,P_k)$ of sets of size at most $n$ is compatible with $G\cup\{1,j\}$. Therefore
$\psi_k^{(G)}(P_1^\tau,P_2,\dots,P_{j-1},\widehat P_j^\tau,P_{j+1},\dots,P_k)\leq \psi_k^{(G\cup\{1,j\})}(n)$.

Setting $n_j:=n/r_1$ if $j\in V_1$ and $n_j:=n$ otherwise, we obtain
\begin{multline*}
\psi_k(P_1,\dots,P_k)\leq\sum_{\tau\in \Upxi}\psi_k^{(G)}(n/r^{\lambda_1},n_2,\dots,n_k)+\sum_{\tau\in \Upxi}\sum_{j\in V_1}\psi_k^{(G\cup \{1,j\})}(n)\leq \\ \leq  O_{d,\varepsilon}\left (r_1^{\lambda_1+\varepsilon} \right )\psi_k^{(G)}(n/r^{\lambda_1},n_2,\dots,n_k)+\sum_{\tau\in \Upxi}\sum_{j\in V_1}\psi_k^{(G\cup\{1,j\})}(n).
\end{multline*}
To bound $\psi_k^{(G)}(n/r_1^{\lambda_1},n_2,\dots,n_k)$, we repeat a similar analysis for $k-1$ more rounds. 

{\bf Round $\pmb{i}$. }Assume that we want to bound $\psi_k^{(G)}(m_1,\dots,m_k)$ with $m_i\leq n$ for every $i\in[k]$, and let $R_1,\dots,R_k$ be sets such that $\psi_k^{(G)}(m_1,\dots,m_k)=\psi_k(R_1,\dots,R_k)$. 

Let $r_i$ be a parameter to be chosen later. By Theorem~\ref{1/r-cutting}, there is a subdivision $\Upxi_i$ of $S_i$ of size $O_{d,\varepsilon}(r_1^{\lambda_i+\varepsilon} )$  which is a $1/r_i$-cutting of each $\Sigma(R_j)$ with $j\in [k]\setminus \{i\},$ and in which  each cell contains at most $n/r_i^{\lambda_i}$ points of $R_i$.
Put $\ell_i := m_i/r_i^{\lambda_i}$, $\ell_j:=m_j$ for $j\in V\setminus (V_i\cup \{i\})$, and $\ell_j:=m_j/r_i$ for $j\in V_i$. Then, with an analysis similar to the one  done in the first round, we obtain
\[ \psi^{(G)}_k(m_1,\dots,m_k)\leq O_{d,\varepsilon}\left (r_i^{\lambda_i+\varepsilon}\right )\psi_k^{(G)}(\ell_1,\dots,\ell_k)+
\sum_{\tau\in\Upxi_i}\sum_{j\in V_i}\psi_k^{(G\cup\{i,j\})}(n).\]

{\bf Wrapping up. }Combining the inequalities of all $k$ rounds, we obtain
\begin{equation*}
\psi_k^{(G)}(n)\leq O_{d,\varepsilon}\left (\Pi_{i\in [k]}r_i^{\lambda_i+\varepsilon}\right )\left (\psi_k^{(G)}\left (N_1,\dots,N_k\right )
+ 
\sum_{i\in [k]}\sum_{j\in V_i}\psi_k^{(G\cup\{i,j\})}(n)\right ),
\end{equation*}
where  \[N_i:=\frac{n}{r_i^{\lambda_i}\prod_{j\in V_i}r_j}.\] 

{\bf The linear optimisation. }
The next step is to find appropriate $r_i$ of the form $r_i=r^{x_i}$ for a sufficiently large constant $r$. We want the term $\psi_k^{(G)}\left (N_1,\dots,N_k\right )$ to be at most $\psi^{(G)}_k(n/r)$,  that is, we want $x_1,\dots,x_k$ to satisfy 
\[\lambda_ix_i+\sum_{j\in V_i}x_j\geq 1 \ \ \ \ \text{for all } i\in[k].\]

This will lead to a linear optimisation problem as follows. (For two vectors $\pmb\lambda, \mathbf x$, $\langle \pmb\lambda,\mathbf x\rangle$ stands for their scalar product, i.e., $\langle \pmb\lambda,\mathbf x\rangle = \sum_{i\in[k]} \lambda_ix_i$.) \\

\begin{align}\label{total}
&\min \langle \pmb\lambda,\mathbf x\rangle \ \ \ \text{subject to: }\\
&x_i\geq 0 \ \ \ \ \ \ \ \ \ \ \ \ \ \, \ \ \    \textrm { for } 1\leq i \leq k\\
\label{cond}
&\lambda_ix_i+\sum_{j\in V_i} x_j\geq 1 \ \  \textrm{ for } 1\leq i \leq k
\end{align}

Denote by $\zeta(G,\pmb{\lambda})$ be the optimum of the linear program above. Then we can bound $\psi_k^{(G)}(n)$  as follows:
\begin{equation}\label{recursion}
\psi_k^{(G)}(n)\leq O_{d,\varepsilon}\left ( r^{\zeta(G,\pmb\lambda)+\varepsilon}\right)\left (\psi_k^{(G)}\left (\frac{n}{r}\right )
+\sum_{i\in [k]}\sum_{j\in V_i}\psi_k^{(G\cup\{i,j\})}(n)\right ).
\end{equation}

\subsection{The recurrence for the diameter case}\label{recdiam}

We will run a similar recurrence scheme as for unit simplices above, but with each time using the cutting result from Theorem \ref{diam cutting} instead of Theorem \ref{1/r-cutting}. Let us briefly explain why Theorem \ref{diam cutting} is applicable in this situation. Indeed, note that since we are working with diameter graphs, for each sphere $\mathbb{S}$ of radius $1$ around a point, every other point from $P$ is inside the ball bounded by $\mathbb{S}$. Thus, in each round we only need to apply a cutting result for the intersection of balls bounded by the unit spheres centred at the points of $P_i$, as there is no point of $P$ outside this intersection. 

We obtain
\begin{equation*}
\phi_k^{(G)}(n)\leq O_{d,\varepsilon}\left (\phi_{i\in [k]}r_i^{\left \lfloor\frac{\lambda_i+1}{2}\right \rfloor+\varepsilon}\right )\left (\phi_k^{(G)}\left (N_1,\dots,N_k\right )
+ 
\sum_{i\in [k]}\sum_{j\in V_i}\phi_k^{(G\cup\{i,j\})}(n)\right ),
\end{equation*}
where  \[N_i:=\frac{n}{r_i^{\left \lfloor\frac{\lambda_i+1}{2}\right \rfloor}\prod_{j\in V_i}r_j}.\]

For a fixed graph $G$ and a profile $\pmb{\lambda}$ let $\xi(G,\mathbf{\lambda})$ be the solution of the following linear optimisation problem.

\begin{align}\label{total2}
&\min \sum_{i\in [k]} \left \lfloor\frac{\lambda_i+1}{2}\right\rfloor x_i \ \ \ \text{subject to: }\\
& \ \ \ \ \ \ \ \ \ \ \ \ \ \ \ \ \ \ \ \ \ \ \ \ \ \ \ \ \ \ x_i\geq 0 \ \ \textrm { for } 1\leq i \leq k\\
\label{cond2}
&\left \lfloor\frac{\lambda_i+1}2\right \rfloor x_i+\sum_{j\,:\,  (i,j)\notin E} x_j\geq 1 \ \  \textrm{ for } 1\leq i \leq k
\end{align}

Then similarly as in the unit distance case, we obtain
\begin{equation}\label{recursion diam}
\phi_k^{(G)}(n)\leq O_{d,\varepsilon}\left ( r^{\xi(G,\pmb\lambda)+\varepsilon}\right)\left (\phi_k^{(G)}\left (\frac{n}{r}\right )
+\sum_{i\in [k]}\sum_{j\in V_i}\phi_k^{(G\cup\{i,j\})}(n)\right ).
\end{equation}

\section{The linear optimisation problem and the proof of Theorem~\ref{unitsimplices}}\label{sec6}
For any two vectors $\mathbf y=(y_1,\ldots, y_t), \mathbf z=(z_1,\ldots, z_t)$, we write $\mathbf y\le \mathbf z$ if $y_i\le z_i$ for each $i=1,\ldots, t$. We write $\mathbf y\lneq \mathbf z$ if, additionally, for some $i$ we have $y_i<z_i$.

We say that a graph $G=([k],E)$ is {\it realisable with profile} $\pmb{\lambda}$ in $\mathbb{R}^d$, if there is a $k$-tuple $(P_1,\dots,P_k)$ with $P_i\subset \R^d$ that has profile 
$\pmb{\lambda}$ and is compatible with $G$.

The following key observation is easy to prove. 

\begin{lemma}\label{spheres} Let $A_1,\dots,A_{\ell}\subseteq \mathbb{R}^d$ be finite sets of points such that $\|a-a'\|=1$ for any \mbox{$a\in A_i, a'\in A_j$} with  $i\ne j$. Then there exist ${\ell}$ spheres $S_1,\dots,S_{\ell}$  that  span pairwise \mbox{orthogonal} flats, satisfying  $A_i\subseteq S_i$. Moreover, if $|A_i|\geq 3$ for some $i$, then $S_i$ is of \mbox{dimension at least $1$.}
\end{lemma}

Applying Lemma~\ref{spheres} with ${\ell}=2$, it follows that if $G$ is realisable with profile  $\pmb{\lambda}$ in $\mathbb{R}^d$, then there are flats $U_1,\dots,U_k$ in $\mathbb{R}^d$ with $\dim U_i=\min\{\lambda_i+1,d\}$ such that $(i,j)\in E$ implies $U_i$ is orthogonal to $U_j$. We  immediately derive the following condition on the weights $\lambda_i$, which was also pointed out in \cite{AS}.

\begin{lemma}\label{lem clique condition}
Let $K=G[V']$ be a clique of size $\ell$ in $G$. Then
\begin{equation}\label{clique condition}
\sum_{i\in V'}\lambda_i\leq d-\ell.
\end{equation}
\end{lemma}

The following lemma shows that, when solving the linear optimization problem, we can potentially drop the assumption that $G$ is realisable with profile $\pmb{\lambda}$ and replace it by another profile $\pmb\lambda'$ with  $\pmb\lambda'\ge \pmb\lambda$.

\begin{lemma}\label{increasing lambda} Let $\pmb{\lambda},\pmb{\lambda'}\in \mathbb{N}^k$ be two profiles  such that $\pmb\lambda\leq \pmb\lambda'$. Then $\zeta(G,\pmb{\lambda})\leq \zeta(G,\pmb{\lambda'})$.
\end{lemma}

\begin{proof} Let $\mathbf{x'}\geq 0$ be an optimal solution for the linear program \eqref{total}--\eqref{cond} defined by $G$ and $\pmb{\lambda'}$.
By setting $x_i=\frac{\lambda_i'}{\lambda_i}x_i'$ for each $i\in [k]$, we obtain a solution $\mathbf{x}$ for the linear program defined by $G$ and $\pmb{\lambda}$ such that $\langle \pmb{\lambda}, \mathbf {x}\rangle=\langle \pmb \lambda' ,\mathbf{x'}\rangle$.
\end{proof}

The next lemma uses induction and allows us to restrict to the case when the profile $\pmb\lambda$ satisfies $\pmb\lambda\ge (2,\ldots,2)$. 
(It can be viewed as a strengthening of Lemma 6.3 from \cite{AS}.)

\begin{lemma}\label{at least 2} Let $(\alpha_{(i,j)})_{1\leq i \leq j \leq d }$ be monotone increasing in both $i$ and $j$ and such that in addition $\alpha_{i,j}+1\leq \alpha_{i+1,j+2}$ for every $i,j\in [d]$. If $G=([k],E)$ is realisable in $\mathbb{R}^d$ with profile $\pmb{\lambda}$ and there is an $i$ for which $\lambda_i\le 1$, then the following holds. If $\psi_{k',d'}(n)=O_{d'} \left ( n^{\alpha_{k',d'}} \right )$ is true for any $(k',d')\lneq (k,d)$, then we have $\psi_{k,d}^{(G)}(n)=O_d\left (n^{\alpha_{k,d}}\right )$.
\end{lemma}

\begin{proof} Without loss of generality, we may assume that $\lambda_1\le 1$. If $\lambda_1 = 0$ then $n\le 2$, and the statement is trivially true. In what follows, we assume that $\lambda_1 = 1.$ 
Fix a $k$-tuple ($P_1,\dots,P_k$) in $\mathbb{R}^d$ of point sets of size $n$ of profile $\pmb{\lambda}$
that is compatible with $G$. We are going to show that $\psi_{k}(P_1,\dots,P_k)=O_d\left (n^{\alpha_d}\right )$. Recall that  we denote by $U_1$ the $2$-flat that contains the circle on which all points from $P_1$ lie.

Let $X\subseteq \bigcup_{i=2}^k P_i$ be the set of those points in $\bigcup_{i=2}^k P_i$ which are at distance $1$  from at most two points of $P_1$, and $Y=\bigcup_{i=2}^k P_i\setminus X$. We denote by $\Psi'_{k}(P_1,\dots,P_k)$ the set of those $k$-vertex simplices that have at least one vertex in $X$ and by $\Psi''_{k}(P_1,\dots,P_k)$ the set of those $k$-vertex simplices that have no vertices in $X$. Then $|\Psi'_{k}(P_1,\dots,P_k)|\leq 2 \psi_{k-1,d}(P_2,\dots,P_k)=O_{d}\left (n^{\alpha_{d}}\right )$. Further, the set $Y$ is contained in the orthogonal complement of the plane $U_1$, thus we have $|\Psi_{k}''(P_1,\dots,P_k)|\leq n \psi_{k-1,d-2}(n)=nO_{d-2}\left (n^{\alpha_{d-2}}\right ) = O_d\left (n^{\alpha_d}\right )$.
Overall, we obtain \[\psi_k(P_1,\dots,P_k)=|\Psi_k'(P_1,\dots,P_k)|+|\Psi''_k(P_1,\dots,P_k)|=O_d\left(n^{\alpha_d}\right ).\qedhere\]
\end{proof}

Finally, the following theorem is the heart of our proof and allows us to  make progress on the problem of bounding $\zeta(k,d)$ in general. 

\begin{theorem}\label{thm main weight} If $G=([k],E)$ is realisable in $\mathbb{R}^d$ with profile $\pmb{\lambda}$ and $\pmb\lambda\ge (2,\ldots, 2)$, then $\zeta(G,\pmb{\lambda})\leq \frac{5}{8}d+\frac{1}{8}k$.
\end{theorem}

\begin{proof} Let $K\subset [k]$ be the largest set of vertices of $G$ that form a clique and that has maximum profile sum, i.e., such that $\sum_{i\in K} \lambda_i$ is maximal. For each $i\in K$, define \[L_i=\setbuilder{v\in V\setminus \{i\}}
{(v,j)\in E \textrm{ for every } j\in K\setminus \{i\}},\] and put $\ell_i=|L_i|$. We define $\mathbf{x}$ for the LP \eqref{total}--\eqref{cond} as follows: set $x_s=\frac{1}{2}$ for $s\in K$, $x_s=0$ for $s\notin K\cup\bigcup_{i\in K} L_i$, and  $x_s=\frac{\lambda_s}{2\lambda_i}\frac{1}{\lambda_i+\ell_i-1}$ for $s\in L_i$. It is clear that all $x_s$ are non-negative.

Let us first show that $\mathbf{x}$ satisfies \eqref{cond}. 
Since $\lambda_i\geq 2$ for all $i\in[k]$, the condition $\eqref{cond}$ holds for $s\in K$. Second, assume that $s\notin K\cup\bigcup_{i\in K} L_i$. Then $s$ is not connected to at least $2$ vertices of $K$, 
and thus  \eqref{cond} holds for such $s$.  Third, due to the fact that $K$ is a clique of maximal size, $L_i$ is an independent set for every $i\in K$. Moreover, since $K$ is of maximal profile sum among cliques of maximal size,  we have $\lambda_s\leq\lambda_i$ for $s\in L_i$. By Lemma~\ref{increasing lambda}, we may without loss of generality assume that 
$\lambda_s=\lambda_i$. In this case, we have
\[x_s=\frac{1}{2}\frac{1}{\lambda_i+\ell_i-1}\ \ \ \  \text{ for }s\in L_i.\]
 With this in mind, for $s\in L_i$ we have \[\lambda_sx_s+\sum_{l\in V_i}x_l\geq x_i+\lambda_sx_s+\sum_{l\in L_i\setminus \{s\}}x_l=\frac{1}{2}+\frac{1}{2}\frac{\lambda_i+\ell_i-1}{\lambda_i+\ell_i-1}=1.\]

Finally, let us show that $\langle\pmb{\lambda},\mathbf{x}\rangle \leq  \frac{5}{8}d+\frac{1}{8}k$. Indeed, using \eqref{clique condition} and the and the AM-GM inequality, we obtain

\begin{align*}\sum_{s\in [k]} \lambda_sx_s\leq& \frac{d-|K|}{2}+\sum_{i\in K}\frac{1}{2}\frac{\lambda_i\ell_i}{\lambda_i+\ell_i-1}\leq \frac{d-|K|}{2}+\frac{1}{2}\sum_{i\in K}\frac{\frac 14(\lambda_i+\ell_i)^2}{\lambda_i+\ell_i-1}\\ \leq& \frac{d}{2}+\frac{1}{8}\sum_{i\in K}\left(\frac{(\lambda_i+\ell_i)^2}{\lambda_i+\ell_i-1}-4\right)
\le \frac{d}{2}+\frac{1}{8}\sum_{i\in K}(\lambda_i+\ell_i)\\
\leq& \frac{d}{2}+\frac{1}{8}(d-|K|+k)\leq \frac{5}{8}d+\frac{1}{8}k. \qedhere
\end{align*}
\end{proof}

\subsection{Proof of Theorem~\ref{unitsimplices}}

\begin{proof}[Proof of Theorem \ref{unitsimplices}]
We will prove that 
\begin{equation}\label{G}
\psi_{k,d}^{(G)}(n) = O_{d,\varepsilon}\big(n^{\frac 58d+\frac 18k+\varepsilon}\big)
\end{equation}
holds for every $\varepsilon$ and every graph on $k$ vertices. Since $\psi_{k,d}(n)\leq \max_{G}\psi_{k,d}^{(G)}(n)$, this implies the statement.

We will prove \eqref{G} by induction on $k,d$ (we refer to it as `outer induction'). The base case is $d=2$, in which the statement trivially holds. Our induction assumption is that the statement holds for any $(k',d')$ with $(k',d')\lneq (k,d)$, any graph $G'$ on $k'$ vertices and for any $\varepsilon>0$. From this assumption we prove that for each graph $G$ on $k$ vertices and any $\varepsilon>0$ we have $\psi_{k,d}^{(G)}(n) = O_{d,\varepsilon}\big(n^{\frac 58d+\frac 18k+\varepsilon}\big)$. For fixed $k,d$ we use `inner' induction over the number of edges in the complement of $G$.

In the base case, we have $\psi_k^{(K_k)}(n)=O_d\left (n^{\lfloor\frac{d}{2}\rfloor}\right )$: indeed, if $P_1,\dots,P_k$ defines the complete graph $K_k$, then we can have $|P_i|\geq 3$ for at most $d/2$ indices $i$ (cf. Lemmas~\ref{spheres} and~\ref{lem clique condition} above). We now prove our statement for a fixed graph $G$, assuming that
\[\psi_{k,d}^{(G')}(n) = O_{d,\varepsilon,G'}\big(n^{\frac 58d+\frac 18k+\varepsilon}\big).\]
holds for any $G'$ with more edges than $G$.

Take any profile $\pmb \lambda$ and $k$-tuple $(P_1,\ldots, P_k)$ of point sets of size $n$ in $\mathbb{R}^d$ with profile $\pmb{\lambda}$ such that $G$ is compatible with $P_1,\ldots, P_k$. If for some $i$ we have $\lambda_i\le 1$ then the outer induction hypothesis (on $k$ and $d$) and Lemma~\ref{at least 2} with $\alpha_{k',d'}=\frac{5}{8}d'+\frac{1}{8}k'+\varepsilon$  implies that $\psi(P_1,\ldots, P_k) = O_{d,\varepsilon}\big(n^{\frac 58d+\frac 18k+\varepsilon}\big).$ Therefore, in what follows we assume that $\pmb\lambda\ge (2,\ldots, 2)$.

Using the notation from Section \ref{rec},  \eqref{recursion} combined with Theorem~\ref{thm main weight} imply
\begin{equation}\label{bigform1}
\psi_k^{(G)}(n)\leq O_{d,\varepsilon}\left ( r^{\frac 58 d+\frac 18k+\varepsilon}\right)\left (\psi_k^{(G)}\left (\frac{n}{r}\right )
+\sum_{i\in [k]}\sum_{j\in V_i}\psi_k^{(G\cup\{i,j\})}(n)\right ).
\end{equation}

By the inner induction, for each $G'=G\cup \{i,j\}$ and for any $\varepsilon_1>0$ we have $\psi_k^{(G')}(n)=O_{d,\varepsilon_1,G'}\big(n^{\frac 58d+\frac 18k+\varepsilon_1}\big)$. This, together with \eqref{bigform1} implies that there is a constant $C=C(d,\varepsilon_1)$ such that
\begin{equation*}
\psi_k^{(G)}(n)\leq O_{d,\varepsilon}\left ( r^{\frac 58 d+\frac 18k+\varepsilon}\right)\left (\psi_k^{(G)}\left (\frac{n}{r}\right )
+C n^{\frac 58d+\frac 18k+\varepsilon_1}\right ).
\end{equation*}

If $r=r(\varepsilon,d)$ is sufficiently large, and $\varepsilon_1$ is sufficiently small, this implies that
\begin{equation*}\label{bigform2}
\psi_k^{(G)}(n)\leq  r^{\frac 58 d+\frac 18k+2\varepsilon}\psi_k^{(G)}\left (\frac{n}{r}\right ) 
\end{equation*}

If $n$ has the form $n = r^{\ell}$ for some integer $\ell$ then the previous inequality, applied $\ell$ times, implies $\psi_k^{(G)} (n)\le r^{\ell({\frac 58 d+\frac 18k+2\varepsilon})}\psi_k^{(G)} (1) =r^{\ell({\frac 58 d+\frac 18k+2\varepsilon})}$, Therefore for any sufficiently large $n$ we have $\psi_k^{(G)} (n)\le n^{{\frac 58 d+\frac 18k+3\varepsilon}}$. Since $\varepsilon>0$ was arbitrary, we may replace  $3\varepsilon$ with $\varepsilon$ in the last inequality. This finishes the proof.
\end{proof}

\section{Proof sketch of Theorem~\ref{diam simplices}}\label{sec8}

We prove a statement similar to Theorem \ref{thm main weight}, but tailored for the diameter case.

\begin{theorem}\label{thm diam weight} If $G=([k],E)$ is realisable in $\mathbb{R}^d$ with profile $\pmb{\lambda}$ and $\pmb\lambda\ge (2,\ldots, 2)$, then \mbox{$\xi(G,\pmb{\lambda})\leq \frac{d}{2}$.}
\end{theorem}

\begin{proof}
Take a maximal clique $K$ in $G$. We define $\mathbf{x}$ as follows. Set $x_i=1$ if $i\in K$ and $x_i=0$ otherwise. With this choice of $\mathbf{x}$,  inequality~\eqref{cond2} holds for every $i\in [k]$. Indeed, it holds trivially if $i\in K$. Otherwise, by maximality, there is at least one vertex of $K$ to which $i$ is not connected, thus $\eqref{cond2}$ holds in this case as well by the assumption $\pmb{\lambda}\geq (2,\dots,2)$.

We also have $\sum_{i\in [k]} \left \lfloor\frac{(\lambda_i+1)}{2}\right \rfloor x_i\leq \sum_{i\in [k]} \frac{(\lambda_i+1)}{2}x_i=\sum_{i\in K} \frac{(\lambda_i+1)}{2}\leq \frac{d}{2}$, where the last inequality holds by~\eqref{clique condition}.
\end{proof}

Now the proof of Theorem \ref{diam simplices} follows the same lines as the proof of Theorem \ref{unitsimplices}, with a few modifications.

\begin{itemize}

\item First, to make the assumption that $\pmb{\lambda}\geq (2,\dots,2)$ we use Theorem \ref{at least 2} with $\alpha_{k',d'}=\frac{d'}{2}+\varepsilon$ instead of $\alpha_{k',d'}=\frac{5}{8}k'+\frac{1}{8}d'+\varepsilon$. (Note that since $\phi_{k,d}(n) \leq \psi_{k,d}(n)$, Lemma \ref{at least 2} is applicable to make the assumption.)

\item Second, instead of $\psi_{k,d}^{(G)}(n)$ we have to bound $\phi_{k,d}^{(G)}(n)$, thus instead of \eqref{recursion}  we use \eqref{recursion diam}.

\item Finally, instead of Theorem~\ref{thm main weight} we use Theorem~\ref{thm diam weight}.

\end{itemize}

After these modifications,  induction and upper bounding steps are the same as in \mbox{Theorem \ref{unitsimplices}.}

\section{Proof of Theorem \ref{nontrivi}}\label{sec10}

First, we prove the only if direction. The construction is very similar to the Lenz configuration. Assume that $G=([k],E)$ is an orthogonality graph. Then there exist $2$-dimensional linear subspaces $V_1,\dots,V_k$ such that if there is an edge between $i$ and $j$ then $V_1$ and $V_2$ are orthogonal. For each $i$ let $S_i$ be a circle of radius $1/\sqrt{2}$ in $S_i$. Then placing $\lfloor n/k \rfloor$ points on each $S_i$, we obtain a set of $n$ points with $\Omega_{\Delta}(n^k)$ copies of $\Delta$.

We now turn to the other direction. Assume that $G=([k], E)$ is not an orthogonality graph. The proof is very similar to the proof of Theorem \ref{unitsimplices}, thus, we only sketch the proof, highlighting the necessary modifications.

We again switch to the $k$-partite version. Using analogous notation as before, we denote by $\psi_{\Delta,d}(P_1,\dots,P_k)$ the number of $\Delta$-copies with the $i$-th vertex in $P_i$.
For any subgraph $G'=([k],E')$ of $G$ we denote by 
{\[\psi_{\Delta,d}^{(G')}(n_1,\dots,n_k):=\max_{|P_i|\leq n_1,\dots,|P_k|\leq n_k}\setbuilder{\psi_{\Delta,d}(P_1,\dots,P_k)}{G'  \textrm{ is compatible with }(P_1,\dots,P_k)},
\]}

Using the cutting result Theorem \ref{1/r-cutting} in $k$ rounds as in Section \ref{rec} to bound $\psi_{\Delta,d}^{(G')}(n,\dots,n)$, the corresponding linear optimisation problem obtained for a given profile $\pmb{\lambda}$ is the following.

\begin{align}\label{totalnontrivi}
&\min \langle \pmb\lambda,\mathbf x\rangle \ \ \ \text{subject to: }\\
&\ \ \ \ \ \ \ \ \ \ \ \, \ \ \ \ \ \ \ \ \ \ x_i\geq 0 \ \    \textrm { for } 1\leq i \leq k\\
\label{nontrivi1}
&\lambda_ix_i+\sum_{(i,j)\in E\setminus E'} x_j\geq 1 \ \  \textrm{ for } 1\leq i \leq k
\end{align}

Denoting by $\zeta_{\Delta}(G',\pmb{\lambda})$ be the optimum of the linear program above, we can bound $\psi_{\Delta,d}^{(G')}(n)=\psi_{\Delta,d}^{(G')}(n,\dots,n)$  as follows:
\begin{equation*}
\psi_{\Delta,d}^{(G')}(n,\dots,n)\leq O_{d,\varepsilon}\left ( r^{\zeta_{\Delta}(G',\pmb\lambda)+\varepsilon}\right)\left (\psi_{\Delta,d}^{(G')}\left (\frac{n}{r}\right )
+\sum_{i\in [k]}\sum_{(i,j)\in E'\setminus E}\psi_{\Delta,d}^{(G'\cup\{i,j\})}(n)\right ).
\end{equation*}

To bound $\psi_{\Delta,d}^{(G')}(n)$ from this recursive inequality, we induct on $|E'\setminus E|$ and on $n$. The base case of the induction is $G'=G$. In this case, since $G$ is not an orthogonality graph in $\mathbb{R}^d$, we obtain by Lemma~\ref{spheres} that if $(P_1,\dots,P_k)$ is compatible with $G$, then at least one $P_i$ is of cardinality at most $2$. Thus, $\psi_{\Delta,d}^{(G'}(n)\leq 2 n^{k-1}=O(n^{k-1})$. 

A version of Theorem \ref{thm main weight} that we need in this proof is the following.

\begin{theorem}\label{thm nontrivi weight}
If a proper subgraph $G'$ of $G$ is realisable in $\mathbb{R}^d$ with profile $\pmb{\lambda}$ and $\pmb{\lambda}\geq (1,\dots,1)$, then $\zeta_{\Delta}(G',\pmb{\lambda})\leq k-\min\big\{1,\frac{4}{d}\big\}$.
\end{theorem}

\begin{proof}
Let $(i,j)\in E'\setminus E$. If $\lambda_i = \lambda_j = 1$, then we can put $x_i = x_j = \frac 12$. Otherwise, set $x_i=\frac{\lambda_j-1}{\lambda_i\lambda_j-1}$, $x_j = \frac{\lambda_i-1}{\lambda_i\lambda_j-1}$. We also set $x_{\ell}=\frac{1}{\lambda_{\ell}}$ for $\ell\neq i$ and obtain in any case that $\mathbf{x}\geq 0$ satisfies \eqref{nontrivi1} and gives $\langle \pmb{\lambda},\mathbf{x} \rangle = k-1$ in the first case and $\langle \pmb{\lambda},\mathbf{x} \rangle=k-2+ \lambda_i\frac{\lambda_j-1}{\lambda_i\lambda_j-1}+\lambda_j\frac{\lambda_i-1}{\lambda_i\lambda_j-1}= k- \frac{\lambda_i+\lambda_j-2}{\lambda_i\lambda_j-1}$ in the second case. It is not difficult to see that the last fraction is minimized for $\lambda_i = \lambda_j=:t$ and is equal to $\frac{2t-2}{t^2-1} = \frac{2}{t+1}$.  We have that $\lambda_i+\lambda_j \le d-2$, and so $t\le \frac d 2-1$, which implies that $\frac 2{t+1}\le \frac 4d$.
\end{proof}

\section{Unit simplices in sets without large complete bipartite graphs}\label{no K_tt}

In this section, we prove Theorem \ref{simplex no K_{t,t}}. The strategy we follow is rather in the spirit of \cite{chains} and \cite{Paths}, instead of \cite{AS}.  We will prove the following slightly more general statement: Assume that $P$ is a set of $n$ points either in $\mathbb{R}^d$ or on a $d$-sphere. If the unit distance graph on $P$ is $K_{t,t}$-free, then the maximum number of unit $(d+1)$-simplices determined by $P$ is $O_{d,\varepsilon,t}(n^{(d+2)/3+\varepsilon})$.

We call a point $q$ \emph{$r$-rich} with respect to a set of points $P$, if there are at least $r$ points in $P$ at unit distance apart from $q$. Theorem \ref{semialg Fox} implies that if $|P|=n$ and $P\subseteq \mathbb{R}^d$, and further $P\cup \{q\}$ contains no $K_{t,t}$ subgraph, then the number of $n^{\alpha}$-rich points with respect to $P$ is 
\begin{equation}\label{alpha rich}
O_{d,\varepsilon,t}(n^{d-(d+1)\alpha+\varepsilon}+n^{1-\alpha+\varepsilon}).
\end{equation}

A simple modification of the proof of Theorem \ref{semialg Fox} implies the same bound for the case when $P$ is contained in a $d$-sphere.

For some fixed $t$, let $U_t(d,n)$ denote the maximum number of unit simplices with $d+1$ vertices $(p_1,\dots,p_{d+1})$ determined  by a set of $n$ points $P$ either in $\mathbb{R}^d$ or on a $d$-sphere. For any fixed $0\leq \alpha \leq 1$ we denote by $R_{\alpha}\subseteq P$ the set of those points that are at least $\frac{1}{2}n^{\alpha}$-rich, but at most $n^{\alpha}$-rich with respect to $P$. Further, let $U_{t,\alpha}(d,n)$ be the subset of those $(d+1)$-simplices in $P$ for which $p_{d+1}\in R_{\alpha}$.
Then we have $U_{t,\alpha}(d,n)\leq \left |R_{\alpha}\right |U_{t}(d-1,n^{\alpha})$.

Using dyadic decomposition we obtain that
\begin{equation*}
U_t(d,n)\leq\sum_{\alpha}U_{t,\alpha}(d,n)\leq \sum_{\alpha}\left |R_{\alpha}\right |U_{t}(d-1,n^{\alpha}),
\end{equation*}
where the sum is taken over logarithmically many $\alpha$'s.

Substituting the bound \eqref{alpha rich} gives
\begin{multline*}U_t(d,n)\leq \log (n) \max_{\alpha}\left \{\left |R_{\alpha}\right |U_{t,\alpha}(d-1,n)\right \}
\\=\max_{\alpha}\left \{O_{d,\varepsilon,t}\left (\min \left \{(n^{d-(d+1)\alpha+\varepsilon}+n^{1-\alpha+\varepsilon}),n \right \} \right)U_{t,\alpha}(d-1,n^{\alpha})\right \}.
\end{multline*}

A careful analysis of this recurrence gives the desired bound. Indeed, the starting case is $d=2$, which is the Spencer-Szemerédi-Trotter bound. For larger $d$ it is not hard to check that the maximum is achieved when $\alpha=\frac{d-1}{d+1}$, in which case the expression on the right hand side can be bounded by
\[O_{d,\varepsilon,t}\left (n^{d-(d-1)+\varepsilon}\right )O_{d,\varepsilon,t}\left( n^{\frac{d-1}{3}+\varepsilon}\right )=O_{d,\varepsilon,n}\left( n^{\frac{d+2}{3}+2\varepsilon}\right ) .\]

\section{Proofs of Theorems~\ref{smallk},~\ref{k=4,d=7} and~\ref{diam-triangle}}\label{sec7}

\begin{proof}[Proof of Theorem~\ref{smallk}] We prove the statement by induction on $d$. The starting cases, $d=4$ and $d=5$, were done in \cite{AS}.

\begin{claim} Let $d\geq 4$. If $G=G([\lfloor d/2 \rfloor+1], E)$ is realisable with profile $\pmb{\lambda}$ in $\mathbb{R}^d$, and the degree of every vertex is strictly less than $\lfloor d/2 \rfloor-1$, then $\zeta(G,\pmb{\lambda})\leq d/2$ if $d$ is even, and $\zeta(G,\mathbf{\lambda})<d/2-1/6$ if $d$ is odd.
\end{claim}

\begin{proof} By Lemma~\ref{increasing lambda} we may assume that $\lambda_i=d$. Set $x_i=\frac{1}{d+2}$ for every $i\in [\lfloor d/2 \rfloor +1]$. This $x$ is a solution to the linear program defined by $G$ and $\pmb{\lambda}$ with $\langle \pmb{\lambda},\mathbf{x} \rangle\leq \frac{d+2}2\cdot \frac d{d+2} = d/2$ if $d$ is even, and with  $\langle \pmb{\lambda}, \mathbf{x} \rangle \leq \frac{d+1}2\cdot \frac d{d+2}  = \frac d2-\frac d{2d+2} <\frac d2-\frac{1}{6}$ if $d$ is odd.
\end{proof}

Thus, without loss of generality we may assume that $G=G([\lfloor d/2 \rfloor+1], E)$ has a vertex of degree at least $\lfloor d/2 \rfloor-1$. We only explain the induction step in the case when $d$ is even. The proof is similar when $d$ is odd.

Let $P_1,\dots, P_{d/2+1}\subseteq \mathbb{R}^d$ be  such that $|P_1|=\dots=|P_{d/2+1}|=n$, and for any $2\leq j < d/2+1$ if $p_1\in P_1$ and $p_j\in P_j$ then $\|p-q\|=1$. Let $S_1$ be a sphere of dimension $\lambda_1$ containing $P_1$, and $S$ a sphere of dimension $d-\lambda_1-1$ containing $P_2,\dots,P_{d/2+1}$. Further, let $U$ be the affine subspace of dimension $d-\lambda_1$ spanned by $S$. We consider $3$ cases depending on the value of $\lambda_1$.

\smallskip

\textbf{Case 1: } $\lambda_1=1$. In this case the bound follows by Lemma \ref{at least 2} with $\alpha_{k',d'}=n^{d'/2+\varepsilon}$ for $d'$ odd and $\alpha_{k',d'}=n^{d'/2-1/6+\varepsilon}$ for $d'$ even. 

\smallskip

\textbf{Case 2: } $\lambda_1=2$. Partition $P_{d/2+1}=Q_1\bigcup Q_2$ into two sets, where $Q_1$ is the set of those points that are at unit distance from every point of $S$ and $Q_2=P_{d/2+1}\setminus Q_1$. Then 
\[\psi_{d/2+1}\left (P_1,\dots, P_{d/2+1}\right )=\psi_{d/2+1}\left (P_1,\dots, P_{d/2},Q_1\right )+\psi_{d/2+1}\left (P_1,\dots, P_{d/2},Q_2\right ).
\]
Since $\dim U=d-3$, by induction we have 
\[\psi_{d/2-1}\left (P_2,\dots, P_{d/2}\right )=O_{d-3,\varepsilon}\left (n^{\frac{(d-3)}{2}-\frac{1}{6}+\varepsilon}\right ).\]
Further, since $Q_1$ and $P_1$ are contained in the same $2$-sphere, we have
{\small \[\psi_{d/2+1}\left (P_1,\dots, P_{d/2},Q_1\right )\leq \psi_2(P_1,Q_1)\psi_{d/2-1}\left (P_2,\dots, P_{d/2}\right )=O_{d-3,\varepsilon}\left (n^{\frac{4}{3}}n^{\frac{(d-3)}{2}-\frac{1}{6}+\varepsilon}\right )=O_{d,\varepsilon}\left (n^{\frac{d}{2}-\frac{1}{3}+\varepsilon}\right ).\]}

On the other hand, any $q\in Q_2$ is at unit distance apart from a sub-sphere of $S$ of dimension at most $d-4$. Thus, for any $q\in Q_2$ we have
\[\psi_{d/2}\left (P_2,\dots, P_{d/2},\{q\}\right )=O_{d-4,\varepsilon}\left (n^{\frac{d}{2}-2+\varepsilon}\right )\]
by induction. This implies
\[\psi_{d/2+1}\left (P_1,\dots,P_{d/2},Q_2\right )\leq |P_1||Q_2|O_{d-4,\varepsilon}\left (n^{\frac{d}{2}-2+\varepsilon}\right )=O_{d,\varepsilon}\left (n^{\frac{d}{2}+\varepsilon}\right ).
\]
\smallskip 

\textbf{Case 3: } $\lambda_1\geq3$. In this case $U$ is of dimension at most $d-4$, thus
\[\psi_{d/2-1} \left (P_2,\dots,P_{d/2}\right )=O_{d-4,\varepsilon}\left (n^{\frac{d}{2}-2+\varepsilon}\right )
\] by induction. This implies
\[\psi_{d/2+1}\left (P_1,\dots,P_{d/2},P_{d/2+1}\right )\leq |P_1||P_{d/2+1}|O_{d-4,\varepsilon}\left (n^{\frac{d}{2}-2+\varepsilon}\right )=O\left (n^{\frac{d}{2}+\varepsilon}\right ).
\]
\end{proof}

\begin{proof}[Proof of Theorem \ref{k=4,d=7}]

The statement follows from the following two claims.

\begin{claim} If $G=([4],E)$ is realisable with $\pmb{\lambda}$ in $\mathbb{R}^7$ and the degree of every vertex of $G$ is less than $2$, then $\zeta(G,\pmb{\lambda})<\frac{10}{3}$.
\end{claim}

\begin{proof} By Lemma~\ref{increasing lambda} we may assume that $\lambda_i=7$ for every $i\in [4]$. Then by setting $x_i=\frac{1}{9}$ for every $i\in [4]$ we obtain a solution for the linear program defined by $G$ and $\mathbf{\lambda}$ with $\langle \pmb{\lambda}\mathbf{x} \rangle =\frac{28}{9}<\frac{10}{3}$.
\end{proof}

\begin{claim} If $G=G([4],E)$ is a graph with a vertex of degree at least $2$, then $\psi_4^{(G)}(n)=O(n^{10/3})$.
\end{claim}

\begin{proof} It is sufficient to prove the following. If $A,B,C,D\subseteq \mathbb{R}^7$ are four point sets such that $|A|=|B|=|C|=|D|=n$ and for any $a\in A$, $b\in B$ and $c\in C$ we have $d(a,b)=d(a,c)=1$, then $\psi\left ( A,B,C,D,\right )=O\left ( n^{10/3}\right)$. Let $S$ be the sphere of the smallest dimension that contains $A$ and $U$ be the affine subspace spanned by $S$. We have a number of cases depending on the dimension of $S$.

\textbf{Case 1: } $S$ is of dimension $1$. Then $B$ and $C$ are contained a subspace $P$ of dimension $5$  orthogonal to $U$. Partition $D=D_1\cup D_2$ into two parts, where $D_1$ contains those points of $D$ that are at distance $1$ apart from at most $2$ points of $A$, and $D_2=D\setminus D_1$. As $D_2\subseteq P$, the number of unit triangles $b,c,d$ for which $b\in B, c\in C, d\in D_2$ by \cite{AS} is $O\left (n^{\frac{7}{3}}\right )$. Thus we obtain the following.
\[\psi_4\left ( A,B,C,D\right )=\psi_4\left ( A,B,C,D_1 \right )+\psi\left ( A,B,C,D_2\right )_4\leq 2|B||C||D_1|+|A|O\left ( n^{\frac{7}{3}}\right )=O\left ( n^{\frac{10}{3}}\right ).\]

\textbf{Case 2: } $S$ is of dimension $2$. Let $T$ be sphere of dimension $3$ whose points are at distance $1$ from every point of $S$. Then $B,C\subseteq T$.

Partition $D=D_1\cup D_2$ into two parts, where $D_1$ is the set of those points in $D$ that are at distance $1$ apart from every point of $T$, and $D_2=D\setminus D_1$. Clearly
\[\psi_4\left ( A,B,C,D\right )=\psi_4\left ( A,B,C,D_1\right )+\psi_4\left ( A,B,C,D_2\right ).
\] Since $A,D_1\subseteq S$, by \cite{sphere} we obtain \[\psi_4\left ( A,B,C,D_1\right )=O\left ( |B| |C||A+D_1|^{\frac{4}{3}}\right )=\left (n^{\frac{10}{3}}\right).\]

For each $d\in D_2$ the set of those points in $B$ and $C$ that are at distance $1$ apart from $d$ are in a $2$-sphere, thus
\[\psi_4\left ( A,B,C,D_2\right )\leq O\left (|A||D_2||B+C|^{\frac{4}{3}}\right )=O\left (n^{\frac{10}{3}}\right ).\]

\textbf{Case 3: } $S$ is of dimension at least $3$. In this case $B$ and $C$ are contained in a sphere of dimension at most $2$, hence the number of pairs $(b,c)$ where $b\in B, c\in C$ and $d(b,c)=1$ is $O\left (n^{\frac{4}{3}}\right)$. Thus
\[\psi_4\left ( A,B,C,D\right )\leq |A||D|O\left ( n^{\frac{4}{3}}\right)=O\left (n^{\frac{10}{3}}\right ).\]
\end{proof}

It is sufficient to prove $\psi_4^{(G)}=O\left (n^{10/3}\right )$ holds for any $G$. We induct on $n$, and on the number of the edges of the complement of $G$. By \eqref{recursion} we have

\[\psi_4^{(G)}(n)\leq O\left (r^{\zeta(G)}\log^{4}\right)\left (\psi_4^{(G)}\left (\frac{n}{r}\right )+\sum_{i\in V}\psi_{3}^{(G_i)}(n)+\sum_{i\ne j,(i,j)\notin E}\psi_4^{(G(i,j))}(n)\right ).\]

Either $G$ has a vertex with degree at least $2$, and then we are done by the second claim, or $\zeta(G,\pmb{\lambda})< \frac{10}{3}$ for any $\pmb{\lambda}$ by the first claim. In this second case we are done by choosing $r$ sufficiently large and using the induction on $\psi_4^{(G)}\left (\frac{n}{r}\right )$ and on $\sum_{i\ne j,(i,j)\notin E}\psi_4^{(G(i,j))}(n)$.
\end{proof}

Note that unlike the proofs of the more general results (such as Theorem \ref{unitsimplices}), the proof about is more specific to unit distances simplices. However, it can easily be modified for bounding the number of congruent simplices. The parts that are specific to unit simplices are in Case 2 and Case 3, where we cite a result from \cite{sphere} according to which the number of maximum number of unit distances determined by a set of $n$ points on a $2$-sphere is $O(n^{4/3})$. In the corresponding situation for general congruent simplices, we need a bipartite variant of this result, in which we count the maximum number of pairs at a fixed distance between two sets of $n$ points, such that the sets are contained in two concentric $2$-spheres. Since both the bipartite version and the original problem from \cite{sphere} can be reduced to bound the number of incidences between a set of $n$ circles of fixed radius and a set of $n$ points in the plane, the bound does not change.

\begin{proof}[Proof of Theorem \ref{diam-triangle}]
The proof is very similar to the proof of Theorem \ref{k=4,d=7}, but with the following versions of the claims, and with using \eqref{recursion diam} instead of \eqref{recursion}.
\begin{claim} If $G=([3],E)$ is the empty graph, then for any profile $\pmb{\lambda}$ we have $\xi(G,\pmb{\lambda})<2$.
\end{claim}
\begin{proof}With a suitable modification of Lemma \ref{increasing lambda} for $\xi$, we may assume that $\lambda_i=5$ for every $i\in [3]$, and set $x_i=\frac{1}{5}$ for every $i\in[3]$.
\end{proof}

\begin{claim}If $G=G([3],E)$ has at least one edge, then $\phi_3^{(G)}(n)=O(n^2)$.
\end{claim}
\begin{proof}
Let $A,B,C\subseteq \mathbb{R}^5$ be sets of $n$ points such that the distance between any two points of $A$ and $B$ is $1$. Then either $\min \{|A|,|B|\}\leq 2$, in which case the bound is trivial, or $|A|,|B|\geq 3$. Then without loss of generality we may assume that $A$ is contained in a $2$-sphere, and $B$ is contained in a circle. Let $C_1\subseteq C$ be the set of those points in $C$ that are at distance $1$ from at most $2$ points of $B$, and let $C_2=C\setminus C_1$. Then $\phi_3(A,B,C_1)=2|C_1||A|=O(n^2)$. Further, $C_2$ is contained in the same $2$-sphere as $A$, thus by \cite{Swanepoel} we have $\phi_3(A,B,C_2)\leq O(n)|B|=O(n^2)$.
\end{proof}
  
\end{proof}

\section{Further results related to Conjecture \ref{conj unit simplices} and the LP}\label{sec LP}

In this final section we prove some results related to the LP that seems to be helpful for making further progress on Conjecture \ref{conj unit simplices}. The only geometric condition that Agarwal and Sharir \cite{AS} imposed on the LP was condition \eqref{clique condition}. We list below a few results that shows that if our ultimate goal is to solve Conjecture \ref{conj unit simplices}, and not the LP itself, we may impose further geometric conditions on the LP.

\begin{lemma}\label{odd cycle} If $G=([k],E)$ is realisable in $\mathbb{R}^d$ with profile $\pmb{\lambda}$ in $\mathbb{R}^d$, and $C$ is a cycle of length $2\ell+1$ in $G$, then \[\sum_{i\in C}\lambda_i\leq \ell d-(2\ell+1).\]
\end{lemma}

\begin{proof} First note that $\lambda_i\leq d-1$ for every $i\in C$. Thus, there are subspaces $U_1,\dots,U_k$ in $\mathbb{R}^d$ such that $U_i$ is of dimension $\lambda_i+1$ and if $(i,j)\in E$ then $U_i$ is orthogonal to $U_j$. Let the vertices of $C$ in the cyclic order be $(a_1,a_2,\dots,a_{\ell},v,b_{\ell},b_{\ell-1},\dots,b_1)$, and for $i\in [\ell]$ let $A_i=U_{a_i}$, $B_i=U_{b_i}$ and $V=U_v$. For each $2\leq i \leq \ell$ we have \[\dim{(A_i\cap B_i)}\leq \dim{A_{i-1}^{\bot}\cap B_{i-1}^{\bot}}\]
and
\[\dim{(A_{i-1}^{\bot}\cap B_{i-1}^{\bot})}=d-\dim A_{i-1}-\dim B_{i-1}+\dim {(A_{i-1}\cap B_{i-1})}.\] These, together with the fact that $C\subseteq A_{\ell}^\bot\cap B_{\ell}^\bot$ holds for every $i\in[k]$ imply \[\dim C\leq \sum_{i\in [\ell]} (d-\dim A_i-\dim B_i).\] 
From these we obtain
\begin{multline*}\sum_{i\in C}\lambda_i=\sum_{i\in C}\dim U_i-1\leq \dim C+\sum_{i\in \ell}(\dim A_i+\dim B_i)-(2\ell+1) \\ 
\leq\sum_{i\in \ell}(d-\dim A_i-\dim B_i)+\sum_{i\in \ell}(\dim A_i+\dim B_i)-(2\ell+1)= \ell d-)-(2\ell+1).
\end{multline*}
\end{proof}

This allows us to prove a strong bound on the solution of the LP under the assumption that the defining graph has large odd-girth.

\begin{theorem}\label{shortest odd} Assume that $G=([k],E)$ is realisable in $\mathbb{R}^d$ with profile $\pmb{\lambda}$. Assume further that the shortest odd cycle $C$ in $G$ has length $2\ell+1$ for $\ell\geq 2$. Then 
\[\zeta(G,\mathbf{\lambda})\leq \frac{\ell}{2\ell-1}d.\]
\end{theorem}

\begin{proof} We define $\mathbf{x}$ as follows. Set $x_i=\frac{1}{2\ell-1}$ if $i\in C$, and $x_i=0$ otherwise. We will show that $\mathbf{x}$ satisfies the inequalities in \eqref{cond} and that $\langle \pmb{\lambda},\mathbf{x}\leq \rangle \frac{\ell}{2\ell-1} d$.  Every vertex $i\in [k]$ is connected to at most $2$ vertices of $C$, otherwise we would find a shorter odd cycle. Thus, for each $i\in [k]$ we have \[\lambda_ix_i+\sum_{j, (i,j)\not\in E} x_i\geq \sum_{j\in C, (i,j)\not\in E} x_i\geq \frac{2\ell-1}{2\ell-1}=1.\]
We also have \[\sum_{i\in [k]}\lambda_ix_i=\sum_{i\in C}\lambda_ix_i=\frac{1}{2\ell-1}\sum_{i\in C}\lambda_i\leq \frac{\ell}{2\ell-1}d-\frac{(2\ell+1)}{2\ell-1},\] where the last equality follows from Lemma~\ref{odd cycle}.
\end{proof}

We also prove a bound in terms of the girth of $G$.

\begin{theorem}\label{girth} Assume that $\pmb{\lambda}\geq 2$. If $G=([k],E)$ is realisable in $\mathbb{R}^d$ with profile $\pmb{\lambda}$ and the girth of $G$ is $2\ell+1$ with $\ell>1$, then \[\zeta(G,\mathbf{\lambda})\leq \frac{d}{2}.\] If the girth of $G$ is $2\ell$ with $\ell>2$, then \[\zeta(G,\mathbf{\lambda})\leq \frac{\ell}{2\ell-1}d.\]
\end{theorem}

\begin{proof} Let $C$ be a shortest cycle in $G$.

\textbf{Proof of first part:} We define $\mathbf{x}$ as follows. Set $x_i=\frac{1}{2\ell}$ if $i\in C$ and $x_i=0$ otherwise. As each vertex $i\in C$ is connected only to its two neighbours in $C$, and $\lambda_i\geq 2$, \eqref{cond} holds for every $i\in C$. Every vertex $i\in [k]\setminus C$ is connected to at most one vertex from $C$, thus \eqref{cond} also holds for every $i\in [k]\setminus C$. Next, it follows from Lemma~\ref{odd cycle} that \[\langle {\pmb{\lambda}\mathbf{x}}\rangle=\sum_i \lambda_ix_i=\frac{1}{2\ell}\sum_{i\in C}\lambda_i\leq \frac{d}{2}.\] This shows that $\zeta(G,\mathbf{\lambda})\leq\frac{d}{2}$.

\smallskip

\textbf{Proof of second part:} We define $\mathbf{x}$ as follows. Set $x_i=\frac{1}{2\ell-1}$ if $i\in C$ and $x_i=0$ otherwise. As $C$ has no chord, $\lambda_i\geq 2$ for every $i\in [k]$, and each vertex $i\notin C$ is connected to at most one vertex of $C$, \eqref{cond} holds for every vertex. Finally, we have  \[\langle \pmb{\lambda}\mathbf{x} \rangle =\sum_i \lambda_ix_i=\frac{1}{2\ell}\sum_{i\in C}\lambda_i\leq \frac{\ell}{2\ell-1}d.\] This shows that $\zeta(G,\mathbf{\lambda})\leq\frac{\ell}{2\ell-1}d$.
\end{proof}

Before the next Lemma, we recall that a set of vectors $W\in \mathbb{R}^d$ is called \emph{almost orthogonal}, if among any three members of $W$ there are two orthogonal. It is known that the maximum cardinality of an almost orthogonal set in $\mathbb{R}^d$ is $2d$, see \cite{chains}.

\begin{lemma} If $G=([k],E)$ is realisable in $\mathbb{R}^d$ with profile $\pmb{\lambda}$ and $K$ is a subgraph of $G$ of independence number $2$, then $\sum_{i\in K}\lambda_i\leq 2d-|K|+1$.
\end{lemma}

\begin{proof} First, assume that there is a $j\in K$ such that $\lambda_j=d$. Then $j$ is an isolated vertex, and $K\setminus \{j\}$ is a clique, thus by \eqref{clique condition} we obtain
\[\sum_{i\in K}=d+\sum_{i\in K\setminus \{j\}}\leq d+(d-|K|+1)=2d-|K|+1
\]

If $\lambda_i\leq d+1$ holds for every $i\in K$, then there are subspaces $U_1,\dots,U_k$ in $\mathbb{R}^d$ such that $U_i$ is of dimension $\lambda_i+1$ and if $(i,j)\in E$ then $U_i$ is orthogonal to $U_j$. For each vertex $i$ of $K$ let $W_i$ be an orthogonal basis of $U_i$. Their union $W=\cup_{i\in K}W_i$ is an almost-orthogonal set, thus $|W|\leq 2d$. This implies 
\[\sum_{i\in K}\lambda_i\leq \sum_{i\in K} (\dim U_i-1)= |W|-|K|\leq 2d-|K| \]
\end{proof}

In general, for the LP we have the following trade-off: If in the corresponding auxiliary graph a vertex $v$ has many non-neighbours, that is helpful, because $v$ gets a lot of "contribution" from other vertices to satisfy its constraint. On the other hand, if a vertex $v$ has many neigbours, then the subspaces it represents is contained in the intersection of many subspaces. Thus, in this situation the $\lambda$ value corresponding to $v$ is usually small. In the next lemma we quantify a trade-off of this type. We will use it to prove Theorem \ref{triangle}.

\begin{lemma}\label{lemma neighbour} If $G=G([k],E)$ is realisable in $\mathbb{R}^d$ with profile $\pmb{\lambda}$ and it has a vertex $i\in[k]$ of degree strictly larger than $d-\lambda_i$, then $\psi_k^{(G)}(n)=0$.
\end{lemma}

\begin{proof} We may assume that the degree of $1$ is larger than $d-\lambda_1$ and its set of neighbours contain \mbox{$\{2,\dots, d-\lambda_1+2\}$.} It is sufficient to prove the following. Let $P_1,\dots, P_k$ be $k$ sets of points in $\mathbb{R}^d$ such that $|P_1|=\dots=|P_k|=n$, and let $U_i$ be the affine subspace spanned by $P_i$. If $\lambda_1=d$, then $1$ cannot have any neighbours, and the statement follows.

Thus, we may assume that $\lambda_1<d$ \mbox{(that is, $P_1$ is contained in a sphere of dimension $\lambda_1$).} Assume that there is a unit-simplex $(p_1,\dots,p_k)\in \Psi_k(P_1,\dots,P_k)$. Then the vertices $p_2,\dots,p_{d-\lambda_1+2}$ form a unit simplex with $d-\lambda_1+1$ vertices. But $\{p_2,\dots,p_{d-\lambda_1+2}\}$ is contained in the orthogonal complement of $U_1$, which is of dimension $d-\lambda_1-1$. This is a contradiction, implying $\psi_k^{(G)}(P_1,\dots,P_k)=0$.
\end{proof}

\begin{lemma}\label{indep} If $G=([k],E)$ is the empty graph, and $\max_i \lambda_i=m$, then $\zeta(G,\pmb{\lambda})\leq \frac{km}{k+m-1}$.
\end{lemma}

\begin{proof} By Lemma~\ref{increasing lambda} we may assume that $\lambda_i=m$ for every $i\in [k]$. We define $\mathbf{x}$ as follows. Set $x_i=\frac{1}{k+m-1}$ for all $i$. With this we have \eqref{cond} holds for every $i\in[k]$ and $\langle \pmb{\lambda}, \mathbf{x}\rangle =\frac{km}{k+m-1}$, thus $\zeta(G,\pmb{\lambda})\leq \frac{km}{k+m-1}$.
\end{proof}

\begin{theorem}\label{triangle} Assume that $\psi_k^G(n)\neq 0$. If $G=([k],E)$ is realisable in $\mathbb{R}^d$ with profile $\mathbf{\lambda}$, and $G$ is triangle-free, then $\zeta(G,\pmb{\lambda})\leq \frac{2}{3}d$.
\end{theorem}

\begin{proof} We may assume that $G$ is bipartite, otherwise we are done by Theorem~\ref{shortest odd}. We distinguish two cases.

\textbf{Case 1:} There is an $i$ such that $\lambda_i\leq \frac{d}{3}$. We define $\mathbf{x}$ as follows. Put $x_i=1$ and let $V=\setbuilder{j}{(i,j)\in E}$. By Lemma \ref{lem clique condition} and Lemma~\ref{increasing lambda} we may assume that $\lambda_j=d-\lambda_i-2$ for every $j\in V$. For $j\in V$ set $x_j=\frac{1}{|V|+(d-\lambda_i-2)-1}$, and for $j\notin V'$ set $x_j=0$. With this choice of $\mathbf{x}$ condition \eqref{cond} holds for vertex $i$. If $j\notin V$, then $j$ is not connected to $i$, thus \eqref{cond} is satisfied for $j$. If $j\in V$ then \[\lambda_j x_j+\sum_{r,(r,j)\notin E}x_r= \frac{|V|-1+(d-\lambda_i-2)}{|V|+(d-\lambda_i-2)-1}=1.\]
Further, we have 
\begin{multline*}\langle \pmb{\lambda}, \mathbf{x} \rangle=\lambda_i+\sum_{j\in V}\lambda_jx_j=\lambda_i+\frac{|V|(d-\lambda_i-2)}{|V|+d-\lambda_i-2-1}
\leq \lambda_i+\frac{1}{4}\frac{(|V|+d-\lambda_i-2)^2}{|V|+d-\lambda_i-3} \\ 
\leq \lambda_i+\frac{1}{4}(|V|+d-\lambda_i)\leq\lambda_i+\frac{1}{4}(d-2\lambda_i+d)\leq \frac{2}{3}d
\end{multline*}
where the first inequality follows from the the inequality between the arithmetic and geometric means, the second is a simple calculation, the third follows from Lemma \ref{lemma neighbour}, and the fourth from the assumption $\lambda_i\leq \frac{d}{3}$.

\smallskip

\textbf{Case 2:} $\lambda_i> \frac{d}{3}$ for every $i\in[k]$. Note that in this case if $\lambda_i\geq  \frac{2}{3}d$, then $i$ is an isolated vertex. Let $L_1$ and $L_2$ be the set of those vertices in the two colour classes for which $\lambda_i<\frac{2}{3}d$, and let $\ell_1=|L_1|$, $\ell_2=|L_2|$. By Lemma~\ref{increasing lambda} we may assume that $\lambda_i=\frac{2}{3}d$ if $i\in L_1\cup L_2$. If $L_1\cup L_2=\emptyset$, then $E=\emptyset$, and we are done by Lemma~\ref{indep}. Otherwise, we define $\mathbf{x}$ as follows. Set $x_i=0$ for all $i$ with $\lambda_i\geq\frac{2}{3}d$. Set $x_i=\frac{1}{\ell_j+\frac{2}{3}d-1}$ if $x_i\in L_j$ for $j=1,2$. A calculation, similar to the one used in the first case, shows that $\sum \lambda_ix_i\leq \frac{2}{3}d$.
Further, by increasing some of the $x_i$'s if necessary, we may assume that equality holds.

The inequalities in \eqref{cond} hold for $i\in L_1$ as \[\lambda_ix_i+\sum_{(i,j)\notin E} x_j\geq \lambda_ix_i+\sum_{j\in L_1} x_j=\frac{\frac{2}{3}d}{\ell_1+\frac{2}{3}d-1}+\frac{\ell_1-1}{\ell_1+\frac{2}{3}d-1}=1.\]
To check it when $i\notin L_1\cup L_2$, in other words when $i$ is an isolated vertex, we have to show that $\sum x_j\geq 1$.  This follows from \[\sum_{j,\lambda_j\leq \frac{2}{3}d}\lambda_jx_j=\frac{2}{3}d.\]
\end{proof}

Finally, we prove that if $k=d$ and there is a solution such that all constraints in \eqref{cond} are satisfied with equality, then the solution to the LP is at most $d/2$.

\begin{theorem}\label{equality}
Let $G=([k],E)$ is realisable in $\mathbb{R}^d$ with profile $\mathbf{\lambda}$. Assume that $\psi_k^G(n)\neq 0$ and that there is an $\mathbf{x}$ for which every constraint in \eqref{cond} is satisfied with equality. Then for $k=d+1$ we have $\zeta(G,\pmb{\lambda})\leq \frac{d}{2}+\frac{1}{2}$.
\end{theorem}

\begin{proof}
Let $\mathbf{x}$ be such that the inequalities in \eqref{cond} hold with equality for every $i\in [k]$. By Lemma \ref{lemma neighbour} every $i\in [k]$ has at most $d-\lambda_i$ neighbours. That is, every $i$ is not connected to at least $k-d+\lambda_i-1$ vertices. This means, that for each $i\in[k]$ there are at least $k-d+\lambda_i-1$ vertices $j\ne i$, where $x_i$ is counted in \[\lambda_jx_j+\sum_{(j,\ell)\notin E}x_j\]

By summing up all inequalities in \eqref{cond}, which hold now with equality, we obtain
\[(k-d-1)\sum_i x_i+\sum_i 2\lambda_ix_i\leq k,\]
and by rearranging
\[\sum_i\lambda_ix_i\leq\frac{1}{2}\left ( k+(d+1-k)\sum_i x_i \right ).\]
This in the case of $k=d+1$ would give $\sum_i\lambda_ix_i\leq \frac{d}{2}+\frac{1}{2}$.

\end{proof}

\vskip+0.2cm

{\sc Acknowledgments. } We would like to thank Alexandr Polyanskii for many interesting discussions related to Theorems~\ref{unitsimplices} and~\ref{diam simplices}, as well as the LP described in Section~\ref{sec6}. We also thank L\'aszl\'o V\'egh for helpful comments on the presentation of the paper.

\section{Appendix 1: Cuttings}
In this section, we overview the proofs of cutting results. As we mentioned earlier, we decided to include this overview because, although most cutting results we use appeared in the literature before, they are rather implicit in some cases.

We start with the probabilistic part of the argument that can be presented in an abstract setting of configuration spaces. In Section~\ref{sec22} we introduce its geometric counterparts, called decomposition schemes. In Section~\ref{sec23} we state and partially derive the results on decomposition schemes in the related geometric scenarios. In Section~\ref{sec24} we put everything together, stating and proving the necessary cutting results.

\subsection{Cuttings in abstract spaces}\label{sec21}
Cuttings, as used in geometric context, give certain structures that are amenable to divide-and-conquer techniques, via probabilistic arguments. It turns out that we can obtain the properties necessary for that while working in 
an abstract framework of {\it configuration spaces}, in which the probabilistic nature of the arguments depend on some formal properties of these spaces. Geometry comes into play later, when we want to make sure that these properties are satisfied. 
Our exposition in this subsection follows Mulmuley \cite{MUL}.

Let $N$ be an abstract universe of objects. A {\it configuration} $\tau$ over $N$ is a pair $(D,L) = (D(\tau),L(\tau))$, where $D$ and $L$ are disjoint. The objects in $D$ are called {\it triggers} (or objects that define $\tau$) and objects in $L$ are called {\it stoppers} (or objects that conflict with $\tau$). The {\it degree $d(\tau)$} is the cardinality of $D(\tau)$ and the {\it level}, denoted $l(\tau)$, is the cardinality of $L(\tau)$. For a positive integer $C$, a {\it $C$-configuration space} $\Pi(N)$ over the universe $N$ is just a multiset of configurations with the following property. 

\begin{itemize}
  \item[(P)] The degree of each configuration in $\Pi(N)$ is at most $C$ and the number of configurations sharing the same trigger set is at most $C$.
\end{itemize}
For any $R\subset N$ and a configuration $\tau =(D,L)$ define $\tau|_R:=(D, L\cap R)$. Also, define $\Pi(R)$ as follows: \[\Pi(R):=\{\tau|_R: \tau\in \Pi(N), D(\tau)\subset R\}.\]

For a subset $R\subset N$, we say that a configuration $\tau$ is {\it active over $R$} if $D(\tau)\subset R$ and $L(\tau)\cap R = \emptyset$, i.e., the configuration $\tau|_R$ has level $0$ in $\Pi(R)$.
\begin{theorem}\label{thmcut1}
  Let $d$ be an integer and let $\Pi(N)$ be any $d$-configuration space, $|N| = n$.  Let $R$ be a random sample of $N$ of size $r$. With probability greater than $1/2$, for each active configuration $\tau$ over $R$ we have $|L(\tau)|\le c\frac nr\log r$ for some absolute constant $c>0$.
\end{theorem}
Most importantly, Theorem~\ref{thmcut1} guarantees the existence of such a sample $R$.

\begin{proof} The statement is only non-trivial for sufficiently large $r$ only, and thus we assume that $r\ge 2d$.
Take a random sample $R$ as in the statement. Let $p(\tau,r)$ denote the conditional probability that $R$ contains no point in conflict with $\tau$, given that it contains the points defining $\tau$. Note that it is simply equal to the fraction of $r$-subsets of $N$ containing $D(\tau)$ that are disjoint with $L(\tau)$. We have
\begin{small}\[p(\tau,r) = \frac{{n-d(\tau)-l(\tau)\choose r-d(\tau)}}{{n-d(\tau)\choose r-d(\tau)}} = \prod_{i=0}^{r-d(\tau)-1}\frac{(n-d(\tau)-l(\tau)-i)}{(n-d(\tau)-i)}\le \Big(\frac{n-l(\tau)}{n}\Big)^{r-d(\tau)}\le e^{-\frac{l(\tau)(r-d(\tau))}{n}}\le e^{-\frac{l(\tau)r}{2n}}.\]
\end{small}
In the last inequality we have used that $r\ge 2d\ge 2d(\tau)$. If we take $c$ from the statement to be bigger than $4d$, then for each $\tau$ with $l(\tau)\ge c\frac nr\log r$ we have
\[p(\tau,r)\le e^{-\frac{4d\frac nrr\log r }{2n}} = e^{-2d\log r} = r^{-2d}.\]

Let $q(\tau,r)$ denote the probability that $R$ contains all elements defining $\tau$. The probability that $\tau$ is active over the random sample $R$ is precisely $p(\tau,r)q(\tau,r)$ by the law of conditional probability. Hence, the probability that some $\tau$ with $l(\tau)\ge \frac{c n\log r}r$ is active over the random sample $R$ is
\[\sum_{\tau\in \Pi: l(\tau)>\frac{c n\log r}r}p(\tau,r)q(\tau,r)\le \sum_{\tau\in \Pi: l(\tau)>\frac{c n\log r}r}r^{-2d} q(\tau,r)\le r^{-2d}\sum_{\tau\in \Pi}q(\tau,r).\]
Note that $\sum_{\tau\in \Pi}q(\tau,r)$ is simply the expected number of the configurations whose trigger sets are fully contained in $R$. Due to (P) the total number of configurations in $\Pi(R)$ is at most $d\sum_{i=0}^d{r\choose i}\le d(r^d+1)<\frac 12 r^{2d}$ for any $r\ge 4$ and $d\ge 1$. Therefore, $\sum_{\tau\in \Pi}q(\tau,r)<\frac 12r^{2d}$ and we get that the last displayed expression is smaller than $\frac 12$.
\end{proof}

{\bf Remark. } We are going to content ourselves with Theorem~\ref{thmcut1} in order to keep the presentation as clean and short as possible. However, in many situations it is possible to get rid of the extra log factor in the applications of this  result by doing a two-stage sampling. Informally, we apply Theorem~\ref{thmcut1} and then apply it again to each of the configurations that have large levels. 

Exactly the same proof allows for the following variant of the theorem.

\begin{theorem}\label{thmcut2}
  Let $d,k$ be integers and let $\Pi_1(N_1),\ldots, \Pi_k(N_k)$ be any $d$-configuration spaces, $|N_i| = n_i$.  Let $R = R_1\cup\ldots\cup R_k$ be a random sample, where for each $i\in[k]$ the set $R_i$ is chosen uniformly at random from all $r$-element subsets of $N_i$. Then with probability greater than $1/2$, for each $i\in[k]$ and each active configuration $\tau$ over $R_i$ we have $|L(\tau)|\le c\frac {n_i}r\log r$ for some absolute constant $c>0$, depending on $d$ and $k$ only.
\end{theorem}

\subsection{Configuration spaces in geometry: decomposition schemes}\label{sec22}
In order to see geometry in the abstract setting of configuration spaces, one should draw the following parallel: \begin{itemize}
\item the universe $N$ is the set of surfaces (or other objects) forming an arrangement
\item configurations are cells to which we restrict in divide-and-conquer
\item for a configuration (cell) $\tau$ the set $D(\tau)$ is its defining set, typically, the set of surfaces that bounds $\tau$ or is a subset of it.
\item  the set $L(\tau)$ is the set of surfaces that crosses $\tau$.
\end{itemize}

Typically, however, we want some extra properties from configuration spaces that we are going to describe below. It seems to be difficult to treat this situation in great generality, therefore, from now on we restrict to the arrangements that are defined by collections of polynomials in $\R^d$. Put $M:=\{f_1,\ldots, f_n\}$ to be one such collection. For each polynomial $f_i$ we denote by $\sigma_i$ the surface determined by $f_i$, and let $N:=\{\sigma_1,\ldots, \sigma_n\}$ 

A {\it sign pattern } for a point $\mathbf x\in \R^d$ is the sequence $\big({\rm sgn}(f_1(\mathbf x)), \ldots, {\rm sgn}(f_n(\mathbf x))\big)$, where for a real number $y$ we have \[{\rm sgn}(y) = \begin{cases}
                                                      1, & \mbox{if } y>0 \\
                                                      0, & \mbox{if } y=0 \\
                                                      -1, & \mbox{if } y<0.
                                                    \end{cases}\]
We say that a set $S\subset \R^d$ {\it has a fixed sign pattern} if all of its points share the same sign pattern.

For a set $Q\subseteq \mathbb{R}^d$ a {\it decomposition scheme} for $N$ restricted to $Q$ is a function $T$ that assigns to each subfamily $R\subset N$ a collection of disjoint semialgebraic sets $T(R)\subseteq \mathbb{R}^d$, such that each member of $T(R)$ has a fixed sign pattern and $Q\subseteq\cup T(R)$. 
For a subset $R\subset N$ and a cell $\Delta\in T(R)$ let $D(\Delta)\subset N$ be the \emph{set of triggers} for $\Delta$, such that each $\sigma\in D(\Delta)$ either contains or is disjoint with $\Delta$. Let $L(\Delta)\subset N$ be the set of stoppers of $\Delta$ that consists of {\it all} $\sigma\in N$ that cross $\Delta$. 

We note that $D(\Delta)$ and $L(\Delta)$ are disjoint, and thus the collections of pairs $(D(\Delta),L(\Delta))$ over all $R$ and $\Delta\in T(R)$ is a configuration space, which we denote by $\Pi(N)$. Formally, we define a map $\chi$ that corresponds to a pair $(R,\Delta)$ with $\Delta\in T(R)$ the pair $(D(\Delta),L(\Delta))$.  Note that the same `geometric' cell $\Delta$ may correspond to different `combinatorial' configurations, i.e., different pairs of sets $(D(\Delta),L(\Delta))$, if it is present in decompositions for different subsets $R_1,R_2\subset N$. This, however, does not affect any of the arguments.

For a constant $C$ and a function $g:\mathbb{R}_{> 0}\to \mathbb{R}_{>0}$,\footnote{we will only use functions of the form $g(r) = r^t$ for some positive $t$} we say that a decomposition scheme $T$ for $Q$ is 
$(C,g(r))$-\emph{canonical} if 
\begin{enumerate}
\item[(C1)] The configuration space $\Pi(N)$ satisfies (P) with $C$.

\item[(C2)] For each $R\subset N$ the collection $T(R)$ of cells maps by $\chi$ bijectively onto the configurations that are active in $\Pi(R)$.
\item[(C3)] For every $R\subset N$ we have $|T(R)|\leq Cg(|R|)$.
\end{enumerate}

We can carry over Theorems~\ref{thmcut1},~\ref{thmcut2} to the setting of decomposition schemes.

\begin{thm}\label{thmdecomp1}
Let $N:=\{\sigma_1,\ldots, \sigma_n\}$ be a collection of polynomial surfaces in $\R^d$, let $Q\subset R^d$ and let $T$ be a $(C,g(r))$-canonical decomposition scheme for $N$ restricted to $Q$ (here, $C>0$ is some constant and $g(r)$ is a function from $\R$ to $\R$). Fix some integer $r\ge 1$ and let $R$ be a random sample from $N$ of size $r$. With probability greater than $1/2$, each cell from $T(R)$ is crossed by at most $c\frac nr\log r$ surfaces from $N$ for some absolute constant $c>0$. Moreover, the total number of cells in $T(R)$ is at most $C g(r).$
\end{thm}
\begin{proof}
The proof is simply by unveiling the definitions.  The part of the statement concerning crossings follows from two facts. First each cell $\Delta\in T(R)$ corresponds to a pair $(D(\Delta),L(\Delta))$ in the configuration space $\Pi(N)$, and $L(\Delta)$ is defined to be the collection of surfaces that cross $\Delta.$ Second, the map $\chi$ between the cells of the decomposition scheme and $\Pi(N)$ obeys property (C2), and thus  Theorem~\ref{thmcut1} applied for $R$ implies that  the cells $\Delta$ from $T(R)$ satisfy the desired bound on the size of $L(\Delta)$.

The bound on the number of cells in $T(R)$ is given by (C3).
\end{proof}
In the same way, we can obtain the following theorem from Theorem~\ref{thmcut2}.
\begin{thm}\label{thmdecomp2}
Let $N_i:=\{\sigma_1^i,\ldots, \sigma_{n_i}^i\}$, $i=1,\ldots, k$, be $k$ collections of polynomial surfaces in $\R^d$, let $Q\subset R^d$ and let 
$T$ be a $(C,g(r))$-canonical decomposition scheme for $N:=N_1\cup\ldots\cup N_k$ restricted to $Q$ (here, $C>0$ is some constant and $g(r)$ is a function from $\R$ to $\R$). Fix some integer $r\ge 1$ and 
let $R = R_1\cup\ldots\cup R_k$ be a random sample, where for each $i\in[k]$ the set $R_i$ is chosen uniformly at random from all $r$-element subsets of $N_i$. Then with probability greater than $1/2$, for each $i\in[k]$ each cell from $T(R)$  is crossed by at most $c\frac {n_i}r\log r$ surfaces from $N$ for some constant $c>0$ depending on $C$ and $k$ only. Moreover, the total number of cells in $T(R)$ is at most $C g(kr).$
\end{thm}

\subsection{Getting canonical decomposition schemes}\label{sec23}

To prove Theorem~\ref{semialg cutting} we need the following decomposition schemes. 

\begin{thm}\label{general scheme} Fix some constant $t$. Let $N = \{\sigma_1,\ldots, \sigma_n\}$ be a collection of algebraic surfaces
of degree at most $t$ in $\mathbb{R}^d$.
\begin{enumerate}
\item \cite{clarkson, Koltun} If $d\le 4$ then for any $\varepsilon>0$ there is a $(C,r^{d+\varepsilon})$-canonical decomposition scheme for $\R^d$  with $C$ depending on $t$ only. 
\item \cite{Koltun} If $d\geq 5$ then for any $\varepsilon>0$ there is a $(C,r^{2d-4+\varepsilon})$-canonical decomposition scheme for $\R^d$ with $C$ depending on $t$ and $d$ only.
\end{enumerate}
\end{thm}

Here, we come to the most tangled part of the story, and several comments are in order. 

Let us first dwell on the $d=2,3$ case. In \cite{clarkson}, the authors provide a cell decomposition that implies an $O_{d,t}(r^{2d-2})$ bound  on $g(r)$ for any $d\ge 2$, and only point out that the bound $O(r^{2d-3+\varepsilon})$ on $g(r)$ for $d\ge 3$ can be obtained analogously, albeit with a fair amount of technical analysis. The verification of the properties (C1), (C2) are not so easy to extract from the text, and we sketch it in the footnote\footnote{To deduce the properties (C1), (C2), we need to get a bit more into the details of \cite{curvecut}. First, for $W\subset V$ on p.98 they define $\phi(W)$ to be the collection of all cells whose defining sets are contained in $W$. In the beginning of p.97 they state that the cells used in the decomposition for a subset $W$ (among those that have their defining sets in $W$) are precisely those who correspond to edge of $H$ that is disjoint with $W$ (in terms of configuration spaces, this is the same as saying that the set $L(\tau)$ for a cell $\tau$ is disjoint from $W$, or that $\tau$ has level $0$ in $\Pi(W)$). Together, this takes care of (C2). Second, on p.98 (between the two displayed formulas) they verify  that, in their notation, the number of cells with the same defining set is a constant. Together with the fact that the cells in the decomposition are, as they call it, Tarski cells, this takes care of (C1) (or (P)).} below.
They also give a cutting result itself in Theorem~4.2, using an abstract result about `frames' that are very similar to decomposition schemes.\footnote{For an interested reader, a {\it frame} is a pair $(H,\phi)$, where $H = (V,E)$ is a hypergraph and $\phi:2^V\to 2^E$ is a map that satisfies the properties (i) $\phi(V)=E$ and (ii) $W'\subseteq W\subseteq V$ implies $\phi(W') \subseteq \phi(W)$. Here, $V$ is the same as $N$ in configuration spaces, and $E$ is the same as the set $L(\sigma)$. The defining (trigger) set is encoded into the definition of $\phi$.}

The case $d\ge 4$ is treated in Koltun's paper \cite{Koltun}. There, the proofs are much sketchier. However, the first thing to observe is that the author uses the same decomposition as Chazelle et al. \cite{curvecut}. In particular, this implies that the resulting structure is a decomposition scheme. For the upper bound on the `complexity' $g(r)$ of the decomposition scheme, the author makes the assumption that the algebraic surfaces are in general position, which is the best thought of as taking a random and tiny perturbation of the defining polynomials which eliminates all possible algebraic dependencies. Koltun's proof gives complexity bounds in this case. For the general case, Koltun refers to the paper of Sharir \cite{sharir}, in which a similar problem was treated. In his paper, Sharir sketched the argument that when bounding from above the complexity of the lower envelope of an arrangement of polynomial surfaces, it is sufficient to deal with the surfaces in general position. Similar general position assumptions are also justified in Chapter 2 of the Handbook of Computational Geometry \cite{AS2} and in \cite{general}. We also only sketch the proofs when assuming general position.

The situation with decomposition schemes for families of spheres is much cleaner and the general position is also easier to handle using simpler perturbation arguments as explained in many parts in Matousek's book \cite{matousek}. In order to obtain them, after a suitable lifting transformation we need to work with decomposition schemes for arrangements of hyperplanes. To get decomposition schemes for arrangements of hyperplanes, we can  use the \emph{bottom-vertex triangulation}, described for example in \cite{matousek}. For the sake of completeness, we explain its construction here.

In order to obtain the bottom-vertex triangulation of a hyperplane arrangement, we first rotate the arrangement so that no two vertices share the same first coordinate. We triangulate each cell of the arrangement separately, and so it is enough to describe the triangulation for a convex polytope $P$. The triangulation is constructed using induction on the dimension. 
For a $2$-dimensional face $F$ of $P$, let $v$ be the vertex of $F$ with the smallest $x_1$ coordinate. Triangulate $F$ by connecting every other vertex of $F$ to $v$. For a face $F$ of dimension larger than $2$ let $v$ be the vertex of $F$ with the smallest $x_1$ coordinate. Triangulate all proper faces of $F$ by induction. The triangulation of $F$ consists of all simplices that can be formed taking the convex hull of $v$ and a simplex in  the triangulation of one the faces of $F$. It is not hard to see that this procedure gives a triangulation of $P$. We think of each simplex in the triangulation as a {\it relatively open} cell in order to make the cells disjoint.

\begin{obs}\label{obs111} If the hyperplanes of the arrangement are in general position, then the total number of simplices in the bottom-vertex triangulation is at most $d!$ times the number of faces (of all dimensions) of $P$.
\end{obs}
\begin{proof} It is easy to show by induction on the dimension that the total number of simplices in the bottom-vertex triangulation is at most $d!$ times the number of faces (of all dimensions) of $P$. Indeed, if this holds for $d-1$ and we need to prove it for $P$ of dimension $d$, then apply induction for each facet, getting that the number $s_i$ of simplices in the bottom-vertex triangulation in the $i$-th face is at most $(d-1)!e_i$, where $e_i$ 
is the number of faces of all dimensions in the $i$-th face. Therefore, the total number of simplices in the bottom-vertex triangulation of $P$ is at most $\sum_i s_i\le \sum_i(d-1)!e_i$. The last sum is at most $d!$ times the number of faces of $P$ since each face of $P$ can be counted at most $d$ times in different $e_i$.
\end{proof}

Below we state and prove two theorems on canonical decomposition schemes for unit distances and diameters. Their deduction is straightforward from the corresponding complexity results, which we also state.

\begin{theorem}\label{surface scheme}Let $H$ be a set of hyperplanes in 
$\mathbb{R}^d$ and $Q\subseteq \mathbb{R}^d$ be an algebraic surface
of dimension $p$ and degree at most $k$. Then there is a $(c,r^{\lfloor (d+p)/2 \rfloor }\log r)$-canonical decomposition scheme $T$ for $H$ restricted to $Q$.
\end{theorem}

In order to prove the theorem above, we use the Generalised Zone Theorem of Aronov et al. as the complexity result. 

\begin{thm}[Generalised Zone Theorem of Aronov et al. \cite{Zone}]\label{zone} Let $0<p<d$ and $\sigma\subseteq \mathbb{R}^d$ be an algebraic surface of dimension $p$ of degree at most $t$, $H$ be a collection of $r$ hyperplanes in $\mathbb{R}^d$. 
Then the total number of faces bounding those cells of the arrangement that intersect 
$\sigma$ is $O_{d,p,t}(r^{\lfloor(d+p)/2 \rfloor}\log r)$.
\end{thm}

\begin{proof}[Proof of Theorem~\ref{surface scheme}]

For each subset $S\subseteq H$ let $T(S)$ be the set of those simplices from the bottom-vertex triangulation of the arrangement of $S$ that intersect $Q$. In order to check (C1), we remark that each simplex in the triangulation is a convex hull of the vertices of a certain cell of the arrangement, and thus can be defined by at most $d(d+1)$ hyperplanes that define that cell. Moreover, the same arrangement of $O_d(1)$ hyperplanes has $O_d(1)$ vertices and thus can define only $O_d(1)$ many different simplices. In order to check property (C2), we see that, first, the trigger set of a cell must be present in order to get that cell and, second, if there is a stopper for a given simplex, i.e., if there is a hyperplane intersecting it, then the vertices of the simplex cannot belong to the same cell of the arrangement. Conversely, it is clear that a simplex appears whenever all its vertices belong to the same cell and thus if it has level 0.

Let us explain why the bound on the number of simplices holds. The simplices intersecting $\sigma$ are coming from those cells that intersect $\sigma$. By Theorem~\ref{zone} the total number of faces bounding those cells that intersect $\sigma$ is $O_{d,p,k}(n^{\lfloor(d+p)/2 \rfloor}\log n)$, thus by Observation \ref{obs111} the bound on the number of simplices follows.
\end{proof}

\begin{theorem}\label{polytope scheme}Let $H$ be a set of hyperplanes in $\mathbb{R}^d$ and $Q$.
Then there is a $(c,r^{\lfloor d/2 \rfloor})$-canonical decomposition scheme $T$ for $H$ restricted to $Q$ with $c=c(d)$.
\end{theorem}

We will need the Upper Bound Theorem.

\begin{thm}[The Upper Bound Theorem for polytopes \cite{upperbound}]
Let $P$ be a polytope in $\mathbb{R}^d$ with $r$ facets. Then the total number of faces of $P$ is $O_d\left (r^{\lfloor d/2\rfloor}\right )$. 
\end{thm}

\begin{proof}[Proof of Theorem~\ref{polytope scheme}]
For a subset $S\subseteq H$ we define $T(S)$ as the simplices in the decomposition of the cell of the arrangement of $S$ that contains $Q$. The triangulation is the same as in the proof of  Theorem~\ref{surface scheme}, and so the properties (C1), (C2) are checked analogously.

To check property (C3), we have to show that the number of simplices in the triangulation is at most $c|S|^{\lfloor r/2 \rfloor}$. Thus follows from combining Observation~\ref{obs111} with the Upper Bound Theorem. 
\end{proof}

\subsection{Putting everything together: cutting results}\label{sec24}
For convenience, we restate the cutting results. The proof scheme is the same for all of the results, and the proof of the last theorem includes all the steps needed to prove the two theorems before. Thus we prove only the last theorem.

\begin{thm} Let $\Sigma$ be a collection of $n$ algebraic surfaces of degree at most $t$ in $\mathbb{R}^d$. Then the following is true for any $\varepsilon>0$ and $1\leq r \leq n$.
\begin{enumerate}
\item[(i)] \cite{curvecut},\cite{Koltun} If $d=2,3,4$ then there is a subdivision $\Xi$ of $\mathbb{R}^d$ of size $O_{t,\eps}\left ( r^{d+\eps}\right)$ which is a $1/r$-cutting $\Xi$ of $\Sigma$. Moreover, for any set of $m$ points $P\subseteq\mathbb{R}^d$ the subdivision $\Xi$ can be chosen so that each cell contains at most $m/r^d$ points of $P$.

\item[(ii)] \cite{Koltun} If $d\geq 5$ then there is a subdivision $\Xi$ of $\mathbb{R}^d$ of size $O_{d,t,\varepsilon}\left (r^{2d-4+\varepsilon}\right )$ which is a $1/r$-cutting of $\Sigma$ of size $O_{d,t,\varepsilon}\left (r^{2d-4+\varepsilon}\right )$. Moreover, for any set of $m$ points $P\subseteq\mathbb{R}^d$ the subdivision $\Xi$ can be chosen so that each cell contains at most $m/r^{2d-4}$ points of $P$.
\end{enumerate}
\end{thm}

\begin{thm}[\cite{AS}] Let $\Sigma_1,\dots,\Sigma_k$ be $k$ sets of spheres in $\mathbb{R}^d$. Then the following statement hold for every $1\leq r \leq n$ and every $\varepsilon>0$.
\begin{enumerate}
\item There is a subdivision $\Xi$ of $\mathbb{R}^d$ of size $O_{d,k,\varepsilon}\left (r^{d+\varepsilon}\right )$ which is a $1/r$-cutting of each $\Sigma_i$. Moreover, for any set of $m$ points $P\subseteq \mathbb{R}^d$ the subdivision $\Xi$ can be chose so that each cell contains at most $m/r^d$ points of $P$.
\item If $\mathbb{S}$ is sphere of dimension $\lambda$ in $\mathbb{R}^d$, then there is  a subdivision of $\mathbb S$ of size $O_{d,k,\varepsilon}\left (r^{\lambda+\varepsilon}\right )$, which is a $1/r$-cutting of each of $\Sigma_i$. Moreover, for any set of $m$ points $P\subseteq \mathbb{S}$ the subdivision $\Xi$ can be chose so that each cell contains at most $m/r^d$ points of $P$. 
\end{enumerate}
\end{thm}

\begin{theorem} Let $\Sigma_1,\dots,\Sigma_k$ be $k$ sets of spheres in $\mathbb{R}^d$ and $B$ be the intersection of balls bounded by the members of $\bigcup_{i\in[k]} \Sigma_i$. Then the following hold for every $1\leq r \leq n$ and $\varepsilon>0$.
\begin{enumerate}
\item There is a subdivision $\Xi$ of $B$ of size $O_{d,k,\varepsilon}\left (r^{\lfloor (d+1)/2 \rfloor+\varepsilon}\right )$, which is a $1/r$-cutting of each $\Sigma_i$. Moreover, for any set of $m$ points $P\subseteq B$ the subdivision $\Xi$ can be chosen so that each cell contains at most $m/r^{\lfloor (d+1)/2 \rfloor}$ points of $P$.

\item If $\mathbb{S}$ is a sphere of dimension $\lambda$ in $\mathbb{R}^d$, then there is subdivision $\Xi$ of the set $\mathbb S\cap B$ of size $O_{d,k,\varepsilon}\left (r^{\lfloor (\lambda+1)/2 \rfloor+\varepsilon}\right )$, which is a $1/r$-cutting of each  $\Sigma_i$. Moreover, for any set of $m$ points $P\subseteq B$ the subdivision $\Xi$ can be chosen so that each cell contains at most $m/r^{\lfloor (\lambda+1)/2 \rfloor}$ points of $P$.
\end{enumerate}
\end{theorem}
\begin{proof}
  We will use the transformation $\varphi: \mathbb{R}^d\to \mathbb{R}^{d+1}$ defined by $\varphi(x_1,\dots,x_d)=(x_1,\dots,x_d,x_1^2+\dots +x_d^2)$ to lift spheres of $\mathbb{R}^d$ to hyperplanes of $\mathbb{R}^{d+1}$. Note that $\varphi$ maps a sphere defined by $x_1^2+\dots +x_d^2=\alpha_1 x_1+\dots +\alpha_d x_d+\beta$ to the hyperplane defined by $x_{d+1}=\alpha_1 x_1+\dots +\alpha_d x_d+\beta$, and it maps a ball bounded by the sphere $x_1^2+\dots +x_d^2=\alpha_1 x_1+\dots +\alpha_d x_d+\beta$ to the half-space bounded by $x_{d+1}=\alpha_1 x_1+\dots +\alpha_d x_d+\beta$. Further note that $\varphi(\mathbb{R}^d)$ is a paraboloid $Q$.
  
  Applying $\varphi$ to the arrangement, we get arrangements $\Sigma'_1,\ldots, \Sigma'_k$ of hyperplanes in $\R^{d+1}$. 
  
   If $k=1$ then we apply  Theorem~\ref{polytope scheme} to $\Sigma':=\Sigma'_1$ with $H:=\varphi(B)$. It gives a $(c,r^{\lfloor (d+1)/2\rfloor})$-canonical decomposition scheme $T$ for $\Sigma'$. Then apply Theorem~\ref{thmdecomp1} to this decomposition scheme with $r^{1+\varepsilon/d}$ and 
   $r$ sufficiently large so that \[c \frac n{r^{1+\varepsilon/d}}\log (r^{1+\varepsilon/d})<\frac nr.\] Then Theorem~\ref{thmdecomp1} guarantees us that a random sample $R$ of size $r^{1+\varepsilon/d}$ leads to a $1/r$-cutting for $\Sigma'$.   
   Moreover, the number of cells in $T(R)$ is $O_{d,\eps}\big((r^{1+\varepsilon/d})^{\lfloor (d+1)/2\rfloor}\big) = O_{d,\eps}(r^{\lfloor (d+1)/2\rfloor+\eps})$.
  
  For $k>1$ the analysis is very similar. We apply  Theorem~\ref{polytope scheme} to $\Sigma':=\Sigma'_1\cup\ldots\cup\Sigma'_k$ with $H:=\varphi(B)$, getting a $(c,r^{\lfloor (d+1)/2\rfloor})$-canonical decomposition scheme $T$ for $\Sigma'$. 
  Then we apply Theorem~\ref{thmdecomp2}.
  Finally, the number of cells is $O_{d,\eps}\big((kr^{1+\varepsilon/d})^{\lfloor (d+1)/2\rfloor}\big) = O_{k,d,\eps}(r^{\lfloor (d+1)/2\rfloor+\eps})$.
  
  Projecting the cutting back to $\R^d$ along the last coordinate, we get the desired result.
  
  For part 2, we note that $\mathbb S$ maps to a flat $\gamma$ of dimension $\lambda$. Restricting to the subspace that is spanned by $\gamma$ and the last coordinate, we get the situation identical to that in part 1, except that the dimension of the ambient space is now $\lambda+1$. The rest of the proof is identical to that in part 1.
\end{proof}


\begin{thebibliography}{100}

\bibitem{similar} P.~K. Agarwal, R. Apfelbaum, G. Purdy, and M. Sharir, {\it Similar simplices in a d-dimensional point set}, In: Proc. 23rd ACM Symposium on Computational Geometry (2007), 232–-238.

\bibitem{Ager} P.~K. Agarwal, J. Erickson, {\it Geometric range searching and its relatives}, In: Contemporary Mathematics 223 (B. Chazelle et al, eds.), American Mathematical Society, Providence (1999), 1–-56. 

\bibitem{AMat} P.~K. Agarwal and J. Matou\v sek, {\it On range searching with semialgebraic sets}, Discrete Comput. Geom. 11 (1994), 393–-418.

\bibitem{AMS} P.~K. Agarwal, J. Matou\v sek, and M. Sharir, {\it On Range Searching with Semialgebraic Sets II}, SIAM J. Comput. 42 (2013), N6, 2039–-2062.


\bibitem{AS2} P.~K. Agarwal and M. Sharir, {\it Arrangements and Their Applications}, In: Handbook of Computational Geometry (J.~R. Sack and J. Urrutia, eds.),  North-Holland, Amsterdam (2000), 49–-119.

\bibitem{AS} P.~K. Agarwal and M. Sharir, {\it The number of congruent simplices in a point set}, Discrete Comput. Geom. 28 (2002), 123–-150.

\bibitem{ARS}N. Alon, L. Rónyai and T. Szabó, { \it Norm-graphs: variations and applications.} J.
Combin. Theory Ser. B, 76(2):280–290, 1999.

\bibitem{Ash} N. Alon and C. Shikhelman, {\it Many $T$ copies in $H$-free graphs}, Journal of Combinatorial Theory Ser. B. 121 (2016), 146--172.

\bibitem{Aps} R. Apfelbaum and M. Sharir, {\it Large complete bipartite subgraphs in incidence graphs of points
and hyperplanes}, SIAM J. Disc. Math. 21 (2007), 707–-725.

\bibitem{Zone} B. Aronov, M. Pellegrini, and M. Sharir. {\it On the zone of a surface in a hyperplane arrangement}, Discrete Comput. Geom. 9 (1993), 177–-186.

%\bibitem{cuttingapp} M. de Berg and O. Schwarzkopf, {\it Cuttings and applications}, Internat. J. Comput. Geom. Appl 5 (1995), 343–-355.

%\bibitem{Bor1} K. Borsuk, \textit{Drei S\"atze \"uber die n-dimensionale euklidische Sph\"are}, Fund. Math. 20 (1933), 177–-190.

\bibitem{general} S. Basu, R. Pollack and M.F Roy, {\it On the number of cells defined by a set of polynomials}, C. R. Acad. Sci. Paris 316 (1993), 573-577.

\bibitem{Bras} P. Brass, {\it Combinatorial geometry problems in pattern recognition},  Discrete Comput.
Geom. 28 (2002), 495–-510.

\bibitem{BK} P. Brass and C. Knauer, {\it On counting point-hyperplane incidences,} Comput. Geom. Theory Appl.
25 (2003), 13–-20.

\bibitem{Bron} H. Br\"onnimann, B. Chazelle, and J. Matou\v sek, {\it Product range spaces, sensitive sampling, and derandomization,} SIAM J. Comput. 28 (1999), 1552–-1575.

\bibitem{BB}B. Bukh, {\it Extremal graphs without exponentially-small bicliques}, (2021) arXiv:2107.04167.

\bibitem{Chaz2} B. Chazelle, {\it Cutting hyperplanes for divide-and-conquer}, Discrete Comput. Geom. 9 (1993),
145–-158.

\bibitem{curvecut} B. Chazelle, H. Edelsbrunner, L.~J. Guibas, and M. Sharir, {\it A singly-exponential stratification scheme for
real semi-algebraic varieties and its applications}, Theoret. Comput. Sci. 84 (1991), 77–-105.

\bibitem{Chaz1} B. Chazelle, M. Sharir, and E. Welzl, {\it Quasi-optimal upper bounds for simplex range
searching and new zone theorems}, Algorithmica 8 (1992), 407–-429.

\bibitem{Clark1} K. L. Clarkson,  {\it New applications of random sampling in computational geometry,}
Discrete Comput. Geom. 2 (1987), 195–-222.

\bibitem{clarkson} K.~L. Clarkson, H. Edelsbrunner, L. Guibas, M Sharir, and E. Welzl, {\it Combinatorial Complexity bounds for arrangements of curves and spheres}, Discrete Comput. Geom. 5 (1990), 99–-160.

\bibitem{repcomp} T. Do, {\it Representation complexity of semi-algebraic graphs},  arXiv preprintarXiv:1709.08259 (2017).

\bibitem{Conlon} D. Conlon, J. Fox, J. Pach, B. Sudakov, and A. Suk, {\it Ramsey-type results for semi-algebraic hypergraph}, Trans. Amer. Math.
Soc 366 (2014).

\bibitem{sphere} P. Erd\H os, D. Hickerson, and J\' anos Pach, {\it A problem of Leo Moser about repeated distances on the sphere}, Amer. Math. Monthly 96 (1989), 569–-575.

\bibitem{Erdosunit} P. Erd\H os, {\it On a set of distances of n points}, Amer. Math. Monthly 53 (1946), 248–-250.

%\bibitem{Lenz} P. Erd\H os, {\it On sets of distances of n points in Euclidean space}, Magyar Tud. Akad. Mat. Kutat\' o Int. K \"ozl. 5 (1960), 165–-169.

\bibitem{ErdosPurdy} P. Erd\H os and G. Purdy, {\it Some extremal problems in geometry IV}, Proc. 7th South-Eastern Conf. Combin.
Graph Theory Comput. (1976) 307–-322.

\bibitem{ES} P. Erd\H os and M. Simonovits, {\it On the chromatic number of geometric graphs}, Ars Combin. 9 (1980), 229–246.


\bibitem{Fox} J. Fox, J. Pach, A. Sheffer, A. Suk, and J. Zahl {\it A semi-algebraic version of Zarankiewicz's problem} J. Eur. Math. Soc. 19 (2017), 1785–-1810.

\bibitem{Paths} N. Frankl and A. Kupavskii, {\it Almost sharp bounds on the number of
discrete chains in the plane}, Combinatorica, to appear, arXiv:1912.00224.


\bibitem{density} J. Fox, J. Pach, A. Suk, {\it Density and regularity theorems for semi-algebraic hypergraphs}, Proceedings of the twenty-sixth
ACM-SIAM symposium on Discrete algorithms 1517–1530

\bibitem{regularity} J. Fox, J. Pach, A. Suk, {\it A polynomial regularity lemma for semialgebraic hypergraphs and its applications in geometry and property testing} SIAM J. Comput. 45 (2016), 2199–2223.

\bibitem{GK} L. Guth and N. H. Katz, {\it On the Erd\H os distinct distances problem in the plane,} Annals of Math. 181 (2015), N1, 155–-190.

%\bibitem{grunbaum}  B. Gr\"unbaum, {\it A proof of V\'azsonyi’s conjecture}, Bull. Res. Council Israel, Sect. A 6 (1956), 77–-78.

\bibitem{HarPel} S. Har-Peled, {\it Geometric approximation algorithms}, American Mathematical Soc. (2011).

\bibitem{epsilon} D. Haussler and E. Welzl, {\it  {$\epsilon$}-nets and simplex range queries}, Discrete Comput. Geom. 2 (1987), 127–-151.

%\bibitem{Heppes} A. Heppes, {\it Beweis einer Vermutung von A. V\'a zsonyi}, Acta Math. Acad. Sci. Hungar. 7 (1956), 463–-466.

\bibitem{Janzer} O. Janzer and A. C. Pohoata, {\it On the Zarankiewicz problem for graphs with bounded VC-dimension}, 	arXiv:2009.00130.



%\bibitem{Kal} G. Kalai, {\it Some old and new problems in combinatorial geometry I: Around Borsuk's problem}, In: Surveys in Combinatorics ( A. Czumaj et al., eds), London Mathematical Society Lecture Note Series, Cambridge University Press, Cambridge (2015), 385–-406.

\bibitem{Kaplan} H. Kaplan, J. Matou\v sek, Z. Safernov\'a, and M. Sharir, {\it Unit distances in three dimensions,} Comb. Probab. Comput. 21 (2012), N4, 597–-610.

 \bibitem{KRS}J. Kollár, L. Rónyai and T. Szabó, {\it Norm-graphs and bipartite Tur´an numbers}
Combinatorica, 16(3):399–406, 1996.

\bibitem{Koltun} V. Koltun, {\it Almost tight upper bounds for vertical decompositions in four dimensions}, J. ACM 51 (2004), 699–-730.

\bibitem{kupavskiidiameter} A.~B. Kupavskii, {\it Diameter graphs in $\mathbf{R}^4$}, Discrete Comput. Geom. 51 (2014), 842–-858

\bibitem{proofschur} A.~B. Kupavskii and A. Polyanskii, {\it Proof of Schur’s conjecture in $\mathbf{R}^d$}, Combinatorica 37 (2017), 1181–-1205.

\bibitem{Mat1} J. Matou\v sek, {\it Cutting hyperplane arrangements,} Discrete Comput. Geom. 6 (1991),
385–-406.

\bibitem{Mat2} J. Matou\v sek, {\it Efficient partition trees}, Discrete Comput. Geom. 8 (1992), 315–-334.

\bibitem{matousek} J. Matou\v sek, {\it Lectures on discrete geometry}, Graduate Texts in Mathematics, vol.212, Springer-Verlag, New York (2002).


\bibitem{upperbound} P. McMullen, {\it The maximum numbers of faces of a convex polytope}, Mathematika 17 (1970), 179–-184.

\bibitem{MUL} K. Mulmuley, {\it Computational Geometry: An Introduction Through
Randomized Algorithms}, Prentice-Hall, (1994)

\bibitem{Mustafapach} N.~H. Mustafa and J. Pach, {\it On the Zarankiewicz problem for intersection hypergraphs}, J. Combin. Theory Ser. A 141 (2016), 1–-7.

\bibitem{MV16} N.~H. Mustafa and K.~Varadarajan,
{\it Epsilon-approximations and epsilon-nets},
In: Handbook of Discrete and Computational Geometry, Third edition (J.~E. Goodman et al, eds.), CRC Press, Boca Raton (2017), 1241–-1267.

\bibitem{PZ} J. Pach and F. De Zeeuw, {\it Distinct distances on algebraic curves in the plane}, Comb. Probab. Comput. 26 (2017), N1, 99–-117.

\bibitem{PA} J. Pach and P.~K. Agarwal, {\it Combinatorial Geometry, Wiley Interscience}, New York (1995).

\bibitem{chains} E.~A. Palsson, S. Senger, and A. Sheffer, {\it On the number of discretec hains}, arXiv preprint arXiv:1902.08259 (2019).

\bibitem{Raz} O. Raz, M. Sharir F. De Zeeuw, {\it Polynomials vanishing on Cartesian products: The Elekes–Szab\'o theorem revisited}, Duke Math.Journal 165 (2016), N18, 3517–-3566.

\bibitem{Ros} M. Rosenfeld, {\it Almost orthogonal lines in $E^d$}
, Applied geometry and discrete
mathematics. The Victor Klee Festschrift (Peter Gritzmann and Bernd Sturmfels,
eds.), DIMACS Ser. Discrete Math. Theoret. Comput. Sci., vol. 4, Amer. Math.
Soc., Providence, RI, 1991, pp. 489–492.

\bibitem{sharir} M. Sharir, {\it Almost tight upper bounds for lower envelopes in higher dimensions}, Discrete
Comput. Geom., 12 (1994), pp. 327–345.

\bibitem{schur} Z. Schur, M.~A. Perles, H. Martini, and Y.~S. Kupitz, {\it On the number of maximal
regular simplices determined by n points in $\mathbf{R}^d$}, In: Discrete and Computational Geometry, The
Goodman-Pollack Festschrift (B. Aronov et al., eds.), Algorithms Combin. 25, Springer, Berlin (2003),
767–-787.

\bibitem{Spencerszemereditrotter} J. Spencer, E. Szemer\'edi, and W. Trotter, Jr., {\it Unit distances in the Euclidean plane} In:  Graph Theory and Combinatorics (B. Bollob\'as ed.), Academic Press, London (1984), 293–-308.

%\bibitem{Stra} S. Straszewicz, {\it Sur un probl\'eme g\'eom\'etrique de P. Erd\H os}, Bull. Acad. Polon. Sci. Cl. III. 5 (1957), 39–-40.

%\bibitem{vazsonyi} K.~J. Swanepoel, {\it A new proof of V\'azsonyi’s conjecture}, J. Combin. Theory Ser.A 115 (2008), 888–-892.

\bibitem{Swanepoel} K.~J. Swanepoel, {\it Unit distances and diameters in Euclidean spaces}, Discrete Comput. Geom. 41 (2009), 1–-27.

\bibitem{unit3dzahl} J. Zahl, {\it An improved bound on the number of point-surface incidences in three dimensions}, Contrib. Discrete Math. 8 (2013), 100–-121.

\bibitem{zahl} J. Zahl, {\it Breaking the 3/2 Barrier for Unit Distances in Three Dimensions}, Int. Math. Res. Notices, rnx336, (2018).



\end{thebibliography}
\end{document}